\documentclass[a4paper,11pt,oneside,reqno]{amsart}
\usepackage[centering,
           top=2.54cm,bottom=2.54cm,right=3.18cm,left=3.18cm,
           headsep=25pt,headheight=20pt]{geometry}
\usepackage{mathrsfs,mathtools}
\usepackage{amsmath,amsthm,amssymb}
\usepackage{comment}
\usepackage[hyperindex]{hyperref}
\hypersetup{bookmarksopen=true,bookmarksopenlevel=1,bookmarksnumbered=true, pdftitle={multiple sums of fractional parts},pdfauthor={},linktoc=page, colorlinks,linkcolor=blue,citecolor=blue,urlcolor=magenta,anchorcolor=green}
\usepackage{cleveref}

\usepackage[all,cmtip]{xy} 
\usepackage{tikz}          
\usepackage{eso-pic}       
\usetikzlibrary{matrix}
\usetikzlibrary{cd}
\usetikzlibrary{arrows}
\usetikzlibrary{shapes.geometric} 
\usepackage{extarrows}    
\usepackage{makecell}     

\numberwithin{equation}{section}
\title{Voiculescu's Theorem in Properly Infinite Factors}
\author{Donald Hadwin}
\curraddr{Department of Mathematics \& Statistics, University of New Hampshire, Durham, 03824, US}
\email{operatorguy@gmail.com}
\author{Minghui Ma}
\curraddr{School of Mathematical Sciences, Dalian University of Technology, Dalian, 116024, China}
\email{minghuima@dlut.edu.cn}
\author{Junhao Shen}
\curraddr{Department of Mathematics \& Statistics, University of New Hampshire, Durham, 03824, US}
\email{junhao.shen@unh.edu}

\date{}

\newtheorem{theorem}{Theorem}[section]
\newtheorem{lemma}[theorem]{Lemma}
\newtheorem{proposition}[theorem]{Proposition}

\newtheorem{conjecture}[theorem]{Conjecture}

\theoremstyle{definition}
\newtheorem{definition}[theorem]{Definition}
\newtheorem{remark}[theorem]{Remark}
\newtheorem{example}[theorem]{Example}

\begin{document}

\clearpage\maketitle
\begin{abstract}
In this paper, we investigate Voiculescu's theorem on approximate unitary equivalence in separable properly infinite factors.
As applications, we establish the norm-denseness of the set of all reducible operators, prove a generalized Voiculescu's bicommutant theorem and a version of asymptotic bicommutant theorem, and obtain an interesting cohomological result.
Additionally, we extend these results to multiplier algebras within separable type $\mathrm{III}$ factors.
At last, a concept of the nuclear length is introduced.
\end{abstract}

\section{Introduction}\label{sec introduction}

\noindent
A famous question concerning the norm-denseness of the set of all reducible operators on a separable complex Hilbert space was raised by P.\,Halmos in \cite[Problem 8]{Hal}.
In order to affirmatively answer this question, D.\,Voiculescu proved a noncommutative Weyl-von Neumann theorem in his groundbreaking paper \cite{Voi}.
This result is now commonly known as Voiculescu's noncommutative Weyl-von Neumann theorem or simply Voiculescu's theorem.
Another significant consequence of Voiculescu's theorem is the relative bicommutant theorem in the Calkin algebra.
Later, W.\,Arveson \cite{Arv} provided an alternative proof of Voiculescu's theorem by using quasicentral approximate units.
He further derived a distance formula for separable norm-closed subalgebras of the Calkin algebra.
Numerous applications of Voiculescu's theorem can be found in Arveson's work \cite{Arv}.
The starting point of this paper is to generalize Voiculescu's theorem in properly infinite factors.
Several results in this direction have been obtained in \cite{CGNN,LSS,SS}.

Throughout this paper, $\mathcal{M}$ denotes a separable properly infinite factor, $\mathcal{K}_{\mathcal{M}}$ is the norm-closed ideal generated by finite projections in $\mathcal{M}$, and $\pi\colon\mathcal{M}\to\mathcal{M}/\mathcal{K}_{\mathcal{M}}$ represents the canonical quotient map.
By definition, $\mathcal{K}_{\mathcal{M}}=\{0\}$ if $\mathcal{M}$ is of type $\mathrm{III}$, $\mathcal{K}_{\mathcal{M}}$ is strong-operator dense in $\mathcal{M}$ if $\mathcal{M}$ is semifinite, and $\mathcal{K}_{\mathcal{M}}$ is the set of all compact operators if $\mathcal{M}$ is of type $\mathrm{I}_\infty$, i.e., $\mathcal{M}\cong\mathcal{B}(\mathcal{H})$.
Given a subalgebra $\mathcal{A}$ of $\mathcal{M}$ with a unit $I_\mathcal{A}$, we say that $\mathcal{A}$ is a unital subalgebra of $\mathcal{M}$ if $I_\mathcal{A}=I$, where $I$ denotes the identity of $\mathcal{M}$.
The set of all nonnegative integers is denoted by $\mathbb{N}$.
In the paper, we will present a proof of the following generalized Voiculescu’s theorem.

\vspace{0.2cm}
T{\scriptsize HEOREM} \ref{thm Voi-into-typeI}.
{\em Let $\mathcal{M}$ be a separable properly infinite factor, $\mathcal{A}$ a separable unital $C^*$-subalgebra of $\mathcal{M}$, and $\mathcal{B}$ a type $\mathrm{I}_{\infty}$ unital subfactor of $\mathcal{M}$.

Then for any unital *-homomorphism $\varphi\colon\mathcal{A}\to\mathcal{B}$ with $\varphi|_{\mathcal{A}\cap\mathcal{K}_{\mathcal{M}}}=0$, there exists a sequence $\{V_k\}_{k\in\mathbb{N}}$ of isometries in $\mathcal{M}\otimes M_2(\mathbb{C})$ such that
\begin{equation*}
  \lim_{k\to\infty}\|(A\oplus\varphi(A))-V_k^*(A\oplus 0)V_k\|=0~\text{for every}~A\in\mathcal{A},
\end{equation*}
and $V_k^*V_k=I\oplus I,V_kV_k^*=I\oplus 0$ for every $k\in\mathbb{N}$.
Furthermore, if $\mathcal{M}$ is semifinite, we can choose $\{V_k\}_{k\in\mathbb{N}}$ such that
\begin{equation*}
  (A\oplus\varphi(A))-V_k^*(A\oplus 0)V_k\in\mathcal{K}_{\mathcal{M}}\otimes M_2(\mathbb{C})~\text{for every}~A\in\mathcal{A}~\text{and}~k\in\mathbb{N}.
\end{equation*}}

Let $\mathcal{M}$ be a separable factor (not necessarily properly infinite).
Recall that an operator $T$ is called reducible in $\mathcal M$ if there is a nontrivial projection $P$ in $\mathcal M$ commuting with $T$, i.e., $PT=TP$ and $P\ne 0,I$.
A striking application of Voiculescu's theorem shows that the set of all reducible operators is norm-dense in $\mathcal{M}$ if $\mathcal{M}$ is a separable type $\mathrm{I}_{\infty}$ factor.
As a consequence of \Cref{thm Voi-into-typeI}, we obtain an affirmative answer to Halmos' Problem 8 in separable properly infinite factors.

\vspace{0.2cm}
T{\scriptsize HEOREM} \ref{thm reducible}.
{\em Let $\mathcal{M}$ be a separable properly infinite factor.
Then the set of all reducible operators is norm-dense in $\mathcal{M}$.}
\vspace{0.2cm}

In a type $\mathrm{I}_n$ factor, i.e., the full matrix algebra $M_n(\mathbb{C})$, Halmos proved in \cite[Proposition 1]{Hal68} that the set of reducible operators is nowhere norm-dense.
Recently, J.\,Shen and R.\,Shi \cite{SS19} proved that in a non-$\Gamma$ type $\mathrm{II}_1$ factor, the set of reducible operators is nowhere norm-dense.
Combining with the above \Cref{thm reducible}, we make the following conjecture, which employs the tools of operator theory to reveal the intrinsic distinction between finite factors and properly infinite factors from a topological perspective.

\begin{conjecture}\label{conj nowhere-dense}
Let $\mathcal{M}$ be a separable factor.
Then the set of reducible operators is nowhere norm-dense in $\mathcal{M}$ if and only if $\mathcal{M}$ is a finite factor.
\end{conjecture}

In \cite{Ped}, G.\,Pedersen posed the question of whether Voiculescu's bicommutant theorem can be extended to general corona algebras.
T.\,Giordano and P.\,Ng \cite{GN} provided a positive answer to Pedersen's question for corona algebras of $\sigma$-unital stable simple and purely infinite $C^*$-algebras.
For recent progress regarding Pedersen's question, we refer the reader to the work of D.\,Kucerovsky and M.\,Mathieu \cite{KM}.
Since $\mathcal{M}/\mathcal{K}_{\mathcal{M}}$ serves as the corona algebra of $\mathcal{K}_{\mathcal{M}}$ when $\mathcal{M}$ is semifinite, we affirmatively resolve Pedersen's question for the specific case of $\mathcal{M}/\mathcal{K}_{\mathcal{M}}$ in the following theorem, which is also a consequence of \Cref{thm Voi-into-typeI}.
Recall that the \emph{relative commutant} of a unital $C^*$-subalgebra $\mathscr{A}$ of $\mathcal{M}/\mathcal{K}_{\mathcal{M}}$ is defined as
\begin{equation*}
  \mathscr{A}^c=\{t\in\mathcal{M}/\mathcal{K}_{\mathcal{M}}\colon
  ta=at\text{ for all } a\in\mathscr{A}\}.
\end{equation*}
The \emph{relative bicommutant} of $\mathscr{A}$ is $\mathscr{A}^{cc}=(\mathscr{A}^c)^c$.

\vspace{0.2cm}
T{\scriptsize HEOREM} \ref{thm bicommutant}.
{\em Let $\mathcal{M}$ be a separable properly infinite semifinite factor.
Then every separable unital $C^*$-subalgebra of $\mathcal{M}/\mathcal{K}_{\mathcal{M}}$ equals its relative bicommutant.}
\vspace{0.2cm}

Note that $\mathcal{K}_{\mathcal{M}}=\{0\}$ and $\mathcal{M}/\mathcal{K}_{\mathcal{M}}=\mathcal{M}$ if $\mathcal{M}$ is a separable type $\mathrm{III}$ factor.
Let $\mathcal{B}$ be a type $\mathrm{I}_\infty$ unital subfactor of $\mathcal{M}$, and $\mathcal{A}=\mathbb{C}I+\mathcal{K}_{\mathcal{B}}$, where $\mathcal{K}_{\mathcal{B}}$ is the ideal of all compact operators in $\mathcal{B}$.
Then the relative bicommutant of $\mathcal{A}$ in $\mathcal{M}$ is equal to $\mathcal{B}$.
Therefore, a version of \Cref{thm bicommutant} does not hold for type $\mathrm{III}$ factors.
Next, we present a different kind of bicommutant theorem.

In \cite{Had}, D.\,Hadwin proved that every separable unital $C^*$-subalgebra of $\mathcal{B}(\mathcal{H})$ equals its approximate bicommutant (see \Cref{def appr-bicommutant} in \Cref{subsec asy-bicommutant} later), where $\mathcal{H}$ is a separable infinite-dimensional complex Hilbert space.
The following theorem is a generalization of Hadwin's asymptotic bicommutant
theorem in separable type $\mathrm{III}$ factors.

\vspace{0.2cm}
T{\scriptsize HEOREM} \ref{thm asy-bicommutant}.
{\em Let $\mathcal{M}$ be a separable type $\mathrm{III}$ factor.
Then every separable unital $C^*$-subalgebra of $\mathcal{M}$ is equal to its relative approximate bicommutant.}
\vspace{0.2cm}

The reason \Cref{thm asy-bicommutant} holds for a factor $\mathcal{M}$ of type $\mathrm{I}_\infty$ or type $\mathrm{III}$ is that the representation theory of $C^*$-subalgebras of $\mathcal{K}_{\mathcal{M}}$ is well-understood (recall that $\mathcal{K}_{\mathcal{M}}=\{0\}$ if $\mathcal{M}$ is a factor of type $\mathrm{III}$).
It is straightforward to prove the asymptotic bicommutant theorem for factors of type $\mathrm{I}_n$.
On proving a version of \Cref{thm asy-bicommutant} for type $\mathrm{II}$ factors, all known techniques fail to be effective.
Thus we need to develop new methods to answer the following conjecture.

\begin{conjecture}\label{conj asy-bicommutant}
Every separable unital $C^*$-subalgebra of a separable factor $\mathcal{M}$ equals its relative approximate bicommutant in $\mathcal{M}$.
\end{conjecture}

It is worth noting that \Cref{conj asy-bicommutant} holds for every separable abelian unital $C^*$-subalgebra of $\mathcal{M}$ by \cite[Theorem 3]{Had14}.
We believe that the above conjecture holds for nuclear $C^*$-subalgebras of $\mathcal{M}$.

Let $\mathcal{M}$ be a separable properly infinite semifinite factor.
S.\,Popa and F.\,Radulescu \cite{PR1988} proved that all derivations of a von Neumann subalgebra of $\mathcal{M}$ into $\mathcal{K}_{\mathcal{M}}$ are inner.
When $\mathcal{M}$ is of type $\mathrm{I}_{\infty}$, J.\,Phillips and I.\,Raeburn \cite{PR} showed that not all derivations of a separable infinite-dimensional $C^*$-subalgebra of $\mathcal{M}$ into $\mathcal{K}_{\mathcal{M}}$ are inner.
For a unital $C^*$-subalgebra $\mathcal{A}$ of $\mathcal{M}$, let $H^1(\mathcal{A},\mathcal{K}_{\mathcal{M}})$ denote the first cohomology group of $\mathcal{A}$ into $\mathcal{K}_{\mathcal{M}}$ (see \Cref{def cohomology} in \Cref{subsec cohomology} later).
The bicommutant $\mathcal{A}''$ is the von Neumann subalgebra of $\mathcal{M}$ generated by $\mathcal{A}$.
Recall that $\pi\colon\mathcal{M}\to\mathcal{M}/\mathcal{K}_{\mathcal{M}}$ represents the canonical quotient map.
As an application of \Cref{thm bicommutant}, we obtain the following result.

\vspace{0.2cm}
T{\scriptsize HEOREM} \ref{thm cohomology}.
{\em Let $\mathcal{M}$ be a separable properly infinite semifinite factor, and $\mathcal{A}$ a separable unital $C^*$-subalgebra of $\mathcal{M}$.

If $\pi(\mathcal{A}'')$ is infinite-dimensional, then $H^1(\mathcal{A},\mathcal{K}_{\mathcal{M}})\ne\{0\}$.}
\vspace{0.2cm}

This paper is structured as follows.
In the next section, we present the fundamental definitions and results.
Our main approximation theorems are provided in \Cref{sec main}, and we establish the proof of generalized Voiculescu's theorem in \Cref{sec Voi}.
Next, we discuss some applications in \Cref{sec application}.
In \Cref{sec multiplier}, we focus on proving analogous results for the multiplier algebras within separable type $\mathrm{III}$ factors.
Finally, in the last section, we introduce a concept of the nuclear length of $C^*$-algebras.

\subsection*{Acknowledgment}

After the current paper was typed up, we learned from our private communication with P.\,Ng that T.\,Giordano, V.\,Kaftal, and P.\,Ng obtained similar results with different proofs, including (i) the noncommutative Weyl-von Neumann theorem for type $\mathrm{II}_{\infty}$ factors, (ii) the bicommutant theorem for type $\mathrm{II}_{\infty}$ factors, and (iii) the asymptotic bicommutant theorem for type $\mathrm{III}$ factors.
They focus on the absorption theorem and extension theory in von Neumann algebras and their proofs employ methods from $C^*$-algebra theory.
Our initial motivation is to explore Halmos' Problem 8 and Voiculescu's theorem in properly infinite factors and the techniques in our proofs originate mainly from von Neumann algebras.
We express gratitude to the anonymous referees for valuable comments and suggestions.
This research was initiated at the University of New Hampshire.

\section{Preliminaries}\label{sec preliminary}

\subsection{Separable Properly Infinite Factors}\label{subsec factors}

Let $\mathcal{H}$ be an infinite-dimensional complex Hilbert space, and $\mathcal{B}(\mathcal{H})$ the algebra consisting of all bounded operators on $\mathcal{H}$.
A selfadjoint unital subalgebra of $\mathcal{B}(\mathcal{H})$ is said to be a \emph{von Neumann algebra} if it is closed in the strong-operator topology.
A von Neumann algebra is considered \emph{separable} if it has a separable predual space (see \cite[Lemma 1.8]{Yam}).
A \emph{factor} is a von Neumann algebra whose center consists of scalar multiples of the identity.

Factors are classified into \emph{finite factors} and \emph{properly infinite factors} determined by a relative dimension function of projections.
Properly infinite factors can be further classified into properly infinite semifinite factors, namely type $\mathrm{I}_{\infty},\mathrm{II}_{\infty}$ factors, and purely infinite factors, namely type $\mathrm{III}$ factors.
For further details, please refer to R.\,Kadison and J.\,Ringrose \cite{KR}.

Throughout this paper, let $\mathcal{M}$ be a separable properly infinite factor.
We denote the identity element of $\mathcal{M}$ by $I_{\mathcal{M}}$ or simply $I$.
Two projections $P$ and $Q$ in $\mathcal{M}$ are said to be \emph{(Murray-von Neumann) equivalent}, denoted by $P\sim Q$, if there is a partial isometry $V$ in $\mathcal{M}$ such that $V^*V=P$ and $VV^*=Q$.
A projection $P$ in $\mathcal{M}$ is said to be \emph{infinite} if it is equivalent to a proper subprojection of $P$ in $\mathcal{M}$.
Otherwise, $P$ is said to be \emph{finite}.
Recall that every nonzero projection is infinite in a type $\mathrm{III}$ factor.

Let $\mathcal{K}_{\mathcal{M}}$ be the norm-closed ideal generated by finite projections in $\mathcal{M}$.
Note that $\mathcal{K}_{\mathcal{M}}=\{0\}$ if $\mathcal{M}$ is of type $\mathrm{III}$, and $\mathcal{K}_{\mathcal{M}}$ is strong-operator dense in $\mathcal{M}$ if $\mathcal{M}$ is semifinite.
Moreover, if $\mathcal{M}$ is of type $\mathrm{I}_{\infty}$, then $\mathcal{M}$ is *-isomorphic to $\mathcal{B}(\mathcal{H}_0)$, where $\mathcal{H}_0$ is a separable infinite-dimensional complex Hilbert space.
In this case, $\mathcal{K}_{\mathcal{M}}$ is the set of all compact operators in $\mathcal{M}$ and $\mathcal{M}/\mathcal{K}_{\mathcal{M}}$ is *-isomorphic to the Calkin algebra.

\subsection{Factorable Maps with Respect to $\mathcal{K}_{\mathcal{M}}$}\label{subsec factorable}

Let $\mathcal{A}$ be a unital $C^*$-subalgebra of $\mathcal{M}$.
Typically, a completely positive map $\psi\colon\mathcal{A}\to\mathcal{M}$ is called \emph{factorable} if $\psi=\eta\circ\sigma$ for some completely positive maps $\sigma\colon\mathcal{A}\to M_n(\mathbb{C})$ and $\eta\colon M_n(\mathbb{C})\to\mathcal{M}$, i.e., the following diagram
\begin{equation*}
    \begin{tikzcd}[column sep =2em, row sep =2em]
        \mathcal{A}  \ar[rr,  "\psi"]\ar[dr,  "\sigma"'] & & \mathcal{M}  \\
        & M_n(\mathbb{C})  \ar[ur,  "\eta"'] &
    \end{tikzcd}
\end{equation*}
commutes.
Furthermore, $\psi$ is said to be \emph{nuclear} if it can be
approximated in the pointwise-norm topology by factorable maps (see \cite[Definition 2.1.1]{BO}).

\begin{definition}\label{def factorable}
Let $\psi\colon\mathcal{A}\to\mathcal{M}$ be a completely positive map with $\psi|_{\mathcal{A}\cap\mathcal{K}_{\mathcal{M}}}=0$.
If $\psi=\eta\circ\sigma$ for some completely positive maps $\sigma\colon\mathcal{A}\to M_n(\mathbb{C})$ and $\eta\colon M_n(\mathbb{C})\to\mathcal{M}$ with $\sigma|_{\mathcal{A}\cap\mathcal{K}_{\mathcal{M}}}=0$, then we say that $\psi$ is \emph{factorable with respect to $\mathcal{K}_{\mathcal{M}}$}.

Let $\mathfrak{F}=\mathfrak{F}(\mathcal{A},\mathcal{M},\mathcal{K}_{\mathcal{M}})$ denote the set of all factorable maps with respect to $\mathcal{K}_{\mathcal{M}}$ from $\mathcal{A}$ into $\mathcal{M}$.
\end{definition}

By definition, the set $\mathfrak{F}$ is a cone.
More precisely, for $j=1,2$, suppose $\psi_j=\eta_j\circ\sigma_j$ for some completely positive maps
\begin{equation*}
  \sigma_j\colon\mathcal{A}\to M_{n_j}(\mathbb{C}),\quad\eta_j\colon M_{n_j}(\mathbb{C})\to\mathcal{M},
\end{equation*}
with $\sigma_j|_{\mathcal{A}\cap\mathcal{K}_{\mathcal{M}}}=0$.
We define completely positive maps as follows:
\begin{equation*}
  \sigma\colon\mathcal{A}\to M_{n_1+n_2}(\mathbb{C}),
  \quad A\mapsto\sigma_1(A)\oplus\sigma_2(A),
\end{equation*}
and
\begin{equation*}
  \eta\colon M_{n_1+n_2}(\mathbb{C})\to\mathcal{M},
  \quad
  \begin{pmatrix}
    X_{11} & X_{12} \\
    X_{21} & X_{22}
  \end{pmatrix}\mapsto\eta_1(X_{11})+\eta_2(X_{22}).
\end{equation*}
Thus, $\psi_1+\psi_2=\eta\circ\sigma\in\mathfrak{F}$.
If $\lambda\geqslant 0$ and $\psi\in\mathfrak{F}$, then clearly $\lambda\psi\in\mathfrak{F}$.
Therefore, $\mathfrak{F}$ is a cone.

\begin{definition}\label{def nuclear}
Let $\widehat{\mathfrak{F}}=
\widehat{\mathfrak{F}}(\mathcal{A},\mathcal{M},\mathcal{K}_{\mathcal{M}})$ denote the closure of $\mathfrak{F}$ in the pointwise-norm topology.
In other words, a map $\varphi\colon\mathcal{A}\to\mathcal{M}$ lies in $\widehat{\mathfrak{F}}$ if for any finite subset $\mathcal{F}$ of $\mathcal{A}$ and any $\varepsilon>0$, there is a map $\psi\in\mathfrak{F}$ such that $\|\varphi(A)-\psi(A)\|<\varepsilon$ for every $A\in\mathcal{F}$.
\end{definition}

By definition, it is straightforward to verify that every map in $\widehat{\mathfrak{F}}$ is completely positive and vanishes on $\mathcal{A}\cap\mathcal{K}_{\mathcal{M}}$.
Maps in $\widehat{\mathfrak{F}}$ are said to be \emph{nuclear with respect to $\mathcal{K}_{\mathcal{M}}$}.

\begin{example}
Let $\varphi\colon\mathcal{A}\to\mathcal{M}$ be a unital *-homomorphism with $\varphi|_{\mathcal{A}\cap\mathcal{K}_{\mathcal{M}}}=0$.
If the inclusion map $\mathrm{id}_{\varphi(\mathcal{A})}\colon\varphi(\mathcal{A})
\hookrightarrow\mathcal{M}$ is nuclear, then the composition $\varphi=\mathrm{id}_{\varphi(\mathcal{A})}\circ\varphi$ is a nuclear map with respect to $\mathcal{K}_{\mathcal{M}}$.
We illustrate the following two examples.

\begin{enumerate}
\item [(1)] Let $\mathcal{A}$ be a nuclear $C^*$-algebra.
    Since $\varphi(\mathcal{A})$ is a nuclear $C^*$-algebra, the inclusion $\mathrm{id}_{\varphi(\mathcal{A})}\colon\varphi(\mathcal{A})
    \hookrightarrow\mathcal{M}$ is automatically nuclear.
\item [(2)] Let $\mathcal{M}$ be an injective factor, and $\mathcal{A}$ an exact $C^*$-algebra.
    Since $\varphi(\mathcal{A})$ is an exact $C^*$-algebra, there exists a nuclear faithful representation $\rho\colon\varphi(\mathcal{A})\to\mathcal{B}(\mathcal{H}_0)$ for some complex Hilbert space $\mathcal{H}_0$.
    By the injectivity of $\mathcal{M}$, the map
\begin{equation*}
  \mathrm{id}_{\varphi(\mathcal{A})}\circ\rho^{-1}
    \colon\rho(\varphi(\mathcal{A}))\to\mathcal{M}
\end{equation*}
extends to a completely positive map $\psi\colon\mathcal{B}(\mathcal{H}_0)\to\mathcal{M}$.
Therefore, $\mathrm{id}_{\varphi(\mathcal{A})}=\psi\circ\rho$ is nuclear.
\end{enumerate}
\end{example}
The reader is referred to N.\,Brown and N.\,Ozawa \cite{BO} for details on nuclear maps and nuclear $C^*$-algebras.
By the following lemma, to define the infinite sum of completely positive maps, it suffices to specify it at the identity element.

\begin{lemma}\label{lem sum-cp}
Let $\{\psi_n\}_{n\in\mathbb{N}}$ be a sequence of completely positive maps from $\mathcal{A}$ into $\mathcal{M}$.
If the series $\sum_{n\in\mathbb{N}}\psi_n(I)$ converges in the strong-operator topology, then $\sum_{n\in\mathbb{N}}\psi_n(A)$ converges in the strong-operator topology for every $A\in\mathcal{A}$.
\end{lemma}

\begin{proof}
Recall that $\mathcal{M}$ acts on a complex Hilbert space $\mathcal{H}$.
Let $x_1,x_2,\ldots,x_k$ be vectors in $\mathcal{H}$, $A\in\mathcal{A}$, and $\varepsilon>0$.
Since $\sum_{n\in\mathbb{N}}\psi_n(I)$ converges in the strong-operator topology, $\|\sum_{n\in\mathbb{N}}\psi_n(I)\|<\infty$ by the uniform boundedness principle \cite[Theorem 1.8.9]{KR}.
Moreover, there exists a natural number $N$ such that for any integers $s\geqslant r\geqslant N$, we have
\begin{equation}\label{equ 1.1Sept29}
  \|A\|^2\Big\|\sum_{n\in\mathbb{N}}\psi_n(I)\Big\|\cdot
  \Big\langle\sum_{n=r}^s\psi_n(I)x_j,x_j\Big\rangle
  <\varepsilon^2~\text{for all}~1\leqslant j\leqslant k.
\end{equation}
By Stinespring's dilation theorem, for any completely positive map $\psi\colon\mathcal{A}\to\mathcal{M}$, we have
\begin{equation}\label{equ 1.2Sept29}
  \|\psi(A)x_j\|^2
  \leqslant\|A\|^2\|\psi(I)\|\langle\psi(I)x_j,x_j\rangle.
\end{equation}
Since the map $\sum_{n=r}^s\psi_n$ is completely positive, it follows from \eqref{equ 1.1Sept29} and \eqref{equ 1.2Sept29} that
\begin{equation*}
  \Big\|\sum_{n=r}^s\psi_n(A)x_j\Big\|<\varepsilon~\text{for all}~1\leqslant j\leqslant k.
\end{equation*}
This completes the proof.
\end{proof}

By \Cref{lem sum-cp}, we are able to define the infinite sum of a sequence of completely positive maps.

\begin{definition}\label{def SF}
Let $\mathfrak{SF}=\mathfrak{SF}(\mathcal{A},\mathcal{M},\mathcal{K}_{\mathcal{M}})$ denote the set of all maps of the form $\sum_{n\in\mathbb{N}}\psi_n$, where $\psi_n\in\mathfrak{F}$ for each $n\in\mathbb{N}$ and the series $\sum_{n\in\mathbb{N}}\psi_n(I)$ converges in the strong-operator topology.

Let $\widehat{\mathfrak{SF}}
=\widehat{\mathfrak{SF}}(\mathcal{A},\mathcal{M},\mathcal{K}_{\mathcal{M}})$ denote the closure of $\mathfrak{SF}$ in the pointwise-norm topology.
\end{definition}

The definition of $\widehat{\mathfrak{SF}}$ is similar to that of $\widehat{\mathfrak{F}}$ in \Cref{def nuclear}.
Since $\mathfrak{F}$ is a subset of $\mathfrak{SF}$, $\widehat{\mathfrak{F}}$ is a subset of $\widehat{\mathfrak{SF}}$.
The following lemma shows that $\widehat{\mathfrak{SF}}$ is closed under countable addition.

\begin{lemma}\label{lem SF-addition}
If $\{\psi_n\}_{n\in\mathbb{N}}$ is a sequence in $\widehat{\mathfrak{SF}}$ such that $\sum_{n\in\mathbb{N}}\psi_n(I)$ converges in the strong-operator topology, then $\sum_{n\in\mathbb{N}}\psi_n\in\widehat{\mathfrak{SF}}$.
\end{lemma}

\begin{proof}
Let $\mathcal{F}$ be a finite subset of $\mathcal{A}$ containing $I$, and $\varepsilon>0$.
For each $n\in\mathbb{N}$, there exists a sequence $\{\psi_{n,m}\}_{m\in\mathbb{N}}$ in $\mathfrak{F}$ such that
\begin{equation*}
  \Big\|\psi_n(A)-\sum_{m}\psi_{n,m}(A)\Big\|<\frac{\varepsilon}{2^{n+1}}~\text{for every}~A\in\mathcal{F}.
\end{equation*}
It follows that
\begin{equation*}
  \Big\|\sum_{n}\psi_n(A)-\sum_{n,m}\psi_{n,m}(A)\Big\|<\varepsilon~\text{for every}~A\in\mathcal{F}.
\end{equation*}
In particular, we have $\lVert\sum_{n}\psi_n(I)-\sum_{n,m}\psi_{n,m}(I)\rVert<\varepsilon$ since $I\in\mathcal{F}$.
Thus, the series $\sum_{n,m}\psi_{n,m}(I)$ converges in the strong-operator topology since $\sum_{n}\psi_n(I)$ converges in the strong-operator topology.
Therefore, $\sum_{n,m}\psi_{n,m}\in\mathfrak{SF}$ and hence $\sum_{n}\psi_n\in\widehat{\mathfrak{SF}}$.
\end{proof}

The following lemma is derived from \cite[Lemma 3.4]{CGNN}.
Recall that $\mathfrak{F}=\mathfrak{F}(\mathcal{A},\mathcal{M},\mathcal{K}_{\mathcal{M}})$ is defined in \Cref{def factorable}.

\begin{lemma}\label{Lemma 3.4}
Let $\mathcal{M}$ be a separable properly infinite factor, $\mathcal{A}$ a unital $C^*$-subalgebra of $\mathcal{M}$, and $P\in\mathcal{K}_{\mathcal{M}}$ a finite projection.

Then every map $\psi\in\mathfrak{F}$ can be approximated in the pointwise-norm topology by maps of the form
\begin{equation*}
  A\mapsto V^*AV,
\end{equation*}
where $V\in\mathcal{M}$ and $PV=0$.
In particular, $V$ can be selected as a partial isometry such that $V^*V=\psi(I)$ when $\psi(I)$ is a projection.
\end{lemma}

\subsection{Cutting down Projections}\label{subsec cut}

In order to facilitate our discussion in subsequent sections, it is necessary to present a set of technical lemmas to cut down infinite projections.
Recall that we write $P\sim Q$ if $P$ and $Q$ are equivalent projections in $\mathcal{M}$.
Moreover, two infinite projections in $\mathcal{M}$ are equivalent by \cite[Corollary 6.3.5]{KR}.

\begin{lemma}\label{lem PAQ=0}
Let $P,Q$ be infinite projections in $\mathcal{M}$, $\rho_1,\rho_2,\ldots,\rho_n$ normal states on $\mathcal{M}$, and $\delta_1,\ldots,\delta_n$ positive numbers such that $\rho_j(Q)>\delta_j>0$ for all $1\leqslant j\leqslant n$.

Then for any operator $A$ in $\mathcal{M}$, there exist infinite projections $P'\leqslant P$ and $Q'\leqslant Q$ in $\mathcal M$ such that $P'AQ'=0$ and $\rho_j(Q')>\delta_j$ for all $1\leqslant j\leqslant n$.
\end{lemma}

\begin{proof}
Consider the polar decomposition $PAQ=VH$, in which $V$ is a partial isometry and $H$ is a positive operator  in $\mathcal M$.
Let $P_0=VV^*\leqslant P$ and $Q_0=V^*V\leqslant Q$.
If $P_0$ is finite, then $P-P_0$ is infinite.
In this case, we set $P'=P-P_0$ and $Q'=Q$.

Now assume that $Q_0(\sim P_0)$ is infinite.
Let $\mathcal{A}$ be a maximal abelian selfadjoint subalgebra of $\mathcal{M}$ which includes $Q_0$ and $H$.
Then there exists a sequence $\{Q_m'\}_{m\in\mathbb{N}}$ of projections in $\mathcal{A}$ such that $Q_0=\sum_mQ_m'$ and $Q_m'\sim Q_0$ in $\mathcal{M}$ for every $m\in\mathbb{N}$.
Since $\rho_j$ is normal, we have
\begin{equation*}
  \sum_m\rho_j(Q_m')=\rho_j(Q_0)<\infty~\text{for all}~1\leqslant j\leqslant n.
\end{equation*}
Therefore, $\rho_j(Q_m')<\rho_j(Q)-\delta_j$ for all $1\leqslant j\leqslant n$ when $m$ is sufficiently large.
We set
\begin{equation*}
  P'=VQ_m'V^*\leqslant P,\quad Q'=Q-Q_m'\leqslant Q.
\end{equation*}
Since $H=HQ=QH$ and $HQ_m'=Q_m'H$, we obtain that
\begin{equation*}
  P'AQ'=VQ_m'V^*VH(Q-Q_m')=VQ_m'Q_0(Q-Q_m')H=0.
\end{equation*}
It is evident that $\rho_j(Q')>\delta_j$ for all $1\leqslant j\leqslant n$.
Furthermore, $P'$ and $Q'$ are infinite projections because $P'\sim Q_m'\sim Q_0$ and $Q'\geqslant Q_{m+1}'\sim Q_0$.
\end{proof}

Let $\mathcal{S}$ be a subset of $\mathcal{M}$.
Then for any operators $X$ and $Y$ in $\mathcal{M}$, we write
\begin{equation*}
  X\mathcal{S}Y=\{XAY\colon A\in\mathcal{S}\}.
\end{equation*}
In particular, $X\mathcal{S}Y=\{0\}$ means that $XAY=0$ for every $A\in\mathcal{S}$.

\begin{lemma}\label{lem PFQ=0}
Let $P,Q$ be infinite projections in $\mathcal{M}$, $\rho_1,\rho_2,\ldots,\rho_n$ normal states on $\mathcal{M}$, and $\delta_1,\ldots,\delta_n$ positive numbers such that $\rho_j(Q)>\delta_j>0$ for all $1\leqslant j\leqslant n$.

Then for any finite subset $\mathcal{F}$ of $\mathcal{M}$, there exist infinite projections $P'\leqslant P$ and $Q'\leqslant Q$  in $\mathcal M$  such that $P'\mathcal{F}Q'=\{0\}$ and $\rho_j(Q')>\delta_j$ for all $1\leqslant j\leqslant n$.
\end{lemma}

\begin{proof}
Let $\mathcal{F}=\{A_1,A_2,\ldots,A_m\}$.
By \Cref{lem PAQ=0}, there exist infinite projections $P_1\leqslant P$ and $Q_1\leqslant Q$ such that $P_1A_1Q_1=0$ and $\rho_j(Q_1)>\delta_j$ for all $1\leqslant j\leqslant n$.
Inductively, we can find infinite projections
\begin{equation*}
  P_m\leqslant P_{m-1}\leqslant\cdots\leqslant P_1\leqslant P,\quad
  Q_m\leqslant Q_{m-1}\leqslant\cdots\leqslant Q_1\leqslant Q
\end{equation*}
such that $P_kA_kQ_k=0$ and $\rho_j(Q_k)>\delta_j$ for all $1\leqslant j\leqslant n$ and $1\leqslant k\leqslant m$.
We set $P'=P_m$ and $Q'=Q_m$.
\end{proof}

\begin{lemma}\label{lem QnFQ0=0}
Let $\{P_n\}_{n\in\mathbb{N}}$ be a sequence of infinite projections in $\mathcal{M}$, and $\{\mathcal{F}_n\}_{n\in\mathbb{N}}$ a sequence of finite subsets of $\mathcal{M}$.

Then there exists a sequence $\{Q_n\}_{n\in\mathbb{N}}$ of infinite projections in $\mathcal{M}$ such that $Q_n\leqslant P_n$ for each $n\geqslant 0$, and $Q_n\mathcal{F}_nQ_0=\{0\}$ for each $n\geqslant 1$.
\end{lemma}

\begin{proof}
Depending on the type of $\mathcal{M}$, the proof splits into the following two cases.

{\bf Case I.}
Suppose $\mathcal{M}$ is a factor of type $\mathrm{I}_\infty$ or $\mathrm{II}_\infty$.
According to \cite[Proposition 8.5.2, Proposition 8.5.5]{KR}, there is a normal faithful semifinite tracial weight $\tau$ on $\mathcal{M}$ such that a projection $E$ in $\mathcal{M}$ is infinite if and only if $\tau(E)=\infty$.
We further assume that $\tau(E)=1$ for every minimal projection $E$ in $\mathcal{M}$ if $\mathcal{M}$ is of type $\mathrm{I}_{\infty}$.

Let $E_1\leqslant P_0$ be a finite projection in $\mathcal{M}$ with $\tau(E_1)=2$, and $\rho_1$ the normal state on $\mathcal{M}$ defined by $\rho_1(A)=\frac{1}{2}\tau(AE_1)$.
Then
\begin{equation*}
  \rho_1(P_0)=\rho_1(E_1)=1>\frac{1}{2}>0.
\end{equation*}
By \Cref{lem PFQ=0}, there exist infinite projections $P_0'\leqslant P_0$ and $Q_1\leqslant P_1$ in $\mathcal{M}$ such that $Q_1\mathcal{F}_1P_0'=\{0\}$ and $\rho_1(P_0')>\frac{1}{2}$, i.e., $\tau(P_0'E_1)>1$.

Let $E_2\leqslant P_0'$ be a finite projection in $\mathcal{M}$ with $\tau(E_2)=3$, and $\rho_2$ the normal state on $\mathcal{M}$ defined by $\rho_2(A)=\frac{1}{3}\tau(AE_2)$.
Then $\rho_1(P_0')>\frac{1}{2}$ and
\begin{equation*}
  \rho_2(P_0')=\rho_2(E_2)=1>\frac{2}{3}>0.
\end{equation*}
Applying \Cref{lem PFQ=0} once again, there exist infinite projections $P_0''\leqslant P_0'$ and $Q_2\leqslant P_2$ in $\mathcal{M}$ such that $Q_2\mathcal{F}_2P_0''=\{0\}$ and $\tau(P_0''E_1)>1$, $\tau(P_0''E_2)>2$.

Continuing this process, for every $n\geqslant 3$, let $E_n\leqslant P_0^{(n-1)}$ be a finite projection in $\mathcal{M}$ with $\tau(E_n)=n+1$, and $\rho_n$ the normal state on $\mathcal{M}$ defined by $\rho_n(A)=\frac{1}{n}\tau(AE_n)$.
Then by \Cref{lem PFQ=0}, there are infinite projections $P_0^{(n)}\leqslant P_0^{(n-1)}$ and $Q_n\leqslant P_n$ in $\mathcal{M}$ such that $Q_n\mathcal{F}_nP_0^{(n)}=\{0\}$ and $\tau(P_0^{(n)}E_k)>k$ for all $1\leqslant k\leqslant n$.

Note that $\{P_0^{(n)}\}_{n\in\mathbb{N}}$ is a decreasing sequence of projections.
Now we set
\begin{equation*}
  Q_0=\bigwedge P_0^{(n)}.
\end{equation*}
Since $\tau$ is normal, we can get $\tau(Q_0E_k)=\lim\limits_{n\to\infty}\tau(P_0^{(n)}E_k)\geqslant k$ for every $k\geqslant 1$, and it follows that $\tau(Q_0)\geqslant\tau(Q_0E_k)\geqslant k$.
We conclude that $\tau(Q_0)=\infty$ and hence $Q_0$ is infinite.

{\bf Case II.}
Suppose $\mathcal{M}$ is a type $\mathrm{III}$ factor.
Then every nonzero projection in $\mathcal{M}$ is infinite.
Moreover, since $\mathcal{M}$ is separable, there is a normal faithful state $\rho$ on $\mathcal{M}$.
Let $\delta$ be a positive number such that $\rho(P_0)>\delta>0$.

By \Cref{lem PFQ=0}, there exist infinite projections $P_0'\leqslant P_0$ and $Q_1\leqslant P_1$ in $\mathcal{M}$ such that $Q_1\mathcal{F}_1P_0'=\{0\}$ and $\rho(P_0')>\delta$.
Similarly, there exist infinite projections $P_0''\leqslant P_0'$ and $Q_2\leqslant P_2$ in $\mathcal{M}$ such that $Q_2\mathcal{F}_2P_0''=\{0\}$ and $\rho(P_0'')>\delta$.
Inductively, for every $n\geqslant 3$, we can find infinite projections $P_0^{(n)}\leqslant P_0^{(n-1)}$ and $Q_n\leqslant P_n$ in $\mathcal{M}$ such that $Q_n\mathcal{F}_nP_0^{(n)}=\{0\}$ and $\rho(P_0^{(n)})>\delta$.

Let $Q_0=\bigwedge P_0^{(n)}$.
Since $\rho$ is normal, we can get $\rho(Q_0)=\lim\limits_{n\to\infty}\rho(P_0^{(n)})\geqslant\delta>0$.
We conclude that $Q_0\ne 0$.
Therefore, $Q_0$ is an infinite projection.
\end{proof}

\begin{lemma}\label{lem QmFQn=0}
Let $\{P_n\}_{n\in\mathbb{N}}$ be a sequence of infinite projections in $\mathcal{M}$, and $\{\mathcal{F}_{m,n}\}_{m,n\in\mathbb{N}}$ a family of finite subsets of $\mathcal{M}$.

Then there exists a sequence $\{Q_n\}_{n\in\mathbb{N}}$ of infinite projections in $\mathcal{M}$ such that $Q_n\leqslant P_n$ for each $n\geqslant 0$, and $Q_m\mathcal{F}_{m,n}Q_n=\{0\}$ when $m\ne n$.
\end{lemma}

\begin{proof}
We can assume that $\mathcal{F}_{n,m}^*=\mathcal{F}_{m,n}$ by replacing $\mathcal{F}_{m,n}$ with $\mathcal{F}_{m,n}\cup\mathcal{F}_{n,m}^*$.
By \Cref{lem QnFQ0=0}, there exist infinite projections $Q_0\leqslant P_0$ and $P_m'\leqslant P_m$ in $\mathcal{M}$ such that $P_m'\mathcal{F}_{m,0}Q_0=\{0\}$ for all $m\geqslant 1$.
Applying \Cref{lem QnFQ0=0} once again, there exist infinite projections $Q_1\leqslant P_1'$ and $P_m''\leqslant P_m'$ in $\mathcal{M}$ such that $P_m''\mathcal{F}_{m,1}Q_1=\{0\}$ for all $m\geqslant 2$.
Inductively, there exist infinite projections $Q_n\leqslant P_n^{(n)}$ and $P_m^{(n+1)}\leqslant P_m^{(n)}$ in $\mathcal{M}$ such that $P_m^{(n+1)}\mathcal{F}_{m,n}Q_n=\{0\}$ for all $m\geqslant n+1$.

Clearly, we have $Q_m\mathcal{F}_{m,n}Q_n=\{0\}$ for all $m>n$.
Furthermore, since $\mathcal{F}_{n,m}^*=\mathcal{F}_{m,n}$, it is obvious that $Q_m\mathcal{F}_{m,n}Q_n=\{0\}$ when $m\ne n$.
\end{proof}

\section{Main Approximation Theorems}\label{sec main}

The following result relies on the concept of \emph{quasicentral approximate units} (see \cite{Arv}) and states that a significant number of completely positive maps from $\mathcal{A}$ into $\mathcal{M}$ lie in the set $\widehat{\mathfrak{SF}}
=\widehat{\mathfrak{SF}}(\mathcal{A},\mathcal{M},\mathcal{K}_{\mathcal{M}})$ as defined in \Cref{def SF}.
A general result will be discussed in \Cref{sec nuc-L}.

\begin{proposition}\label{prop into-typeI}
Let $\mathcal{M}$ be a separable properly infinite factor, $\mathcal{A}$ a unital $C^*$-subalgebra of $\mathcal{M}$, and $\mathcal{B}$ a type $\mathrm{I}_{\infty}$ unital subfactor of $\mathcal{M}$.

Then $\psi\in\widehat{\mathfrak{SF}}$ for every completely positive map $\psi\colon\mathcal{A}\to\mathcal{B}$ with $\psi|_{\mathcal{A}\cap\mathcal{K}_{\mathcal{M}}}=0$.
\end{proposition}

\begin{proof}
Let $\mathcal{F}$ be a finite subset of $\mathcal{A}$ containing $I$, and $\varepsilon>0$.
According to \cite[Theorem 2]{Arv}, there exists a sequence $\{E_n\}_{n\in\mathbb{N}}$ of finite rank positive operators in $\mathcal{B}$ such that $\sum_nE_n^2=I$ and
\begin{equation*}
  \Big\|\psi(A)-\sum_nE_n\psi(A)E_n\Big\|<\varepsilon~\text{for every}~A\in\mathcal{F}.
\end{equation*}
For each $n\in\mathbb{N}$, let $P_n$ denote the finite rank projection $R(E_n)$ in $\mathcal{B}$.
Since $P_n\mathcal{B}P_n$ is *-isomorphism to a full matrix algebra, we can construct a map $\psi_n\in\mathfrak{F}$ by
\begin{equation*}
  \psi_n\colon\mathcal{A}\to P_n\mathcal{B}P_n,\quad A\mapsto E_n\psi(A)E_n.
\end{equation*}
It is clear that $\|\psi(I)-\sum_{n}\psi_n(I)\|<\varepsilon$ since $I\in\mathcal{F}$.
Consequently, the series $\sum_n\psi_n(I)$ converges in the strong-operator topology.
Therefore, $\sum_n\psi_n\in\mathfrak{SF}$ and it follows that $\psi\in\widehat{\mathfrak{SF}}$.
\end{proof}

\begin{remark}\label{rem into-typeI}
As a consequence, if $\mathcal{M}$ is of type $\mathrm{I}_\infty$, then $\psi\in\widehat{\mathfrak{SF}}$ for every completely positive map $\psi\colon\mathcal{A}\to\mathcal{M}$ with $\psi|_{\mathcal{A}\cap\mathcal{K}_{\mathcal{M}}}=0$.
\end{remark}

U.\,Haagerup \cite{Haa} proved that every completely positive map from a finite-dimensional unital subfactor of $\mathcal{M}$ into $\mathcal{M}$ can be expressed in the form $B\mapsto T^*BT$.
Utilizing Haagerup's result, we are now able to demonstrate our main approximation theorem.

\begin{theorem}\label{thm main}
Let $\mathcal{M}$ be a separable properly infinite factor, $\mathcal{A}$ a unital $C^*$-subalgebra of $\mathcal{M}$, and $P\in\mathcal{K}_{\mathcal{M}}$ a finite projection.

Then any $\psi\in\widehat{\mathfrak{SF}}$ can be approximated in the pointwise-norm topology by maps of the form
\begin{equation*}
  A\mapsto V^*AV,
\end{equation*}
where $V\in\mathcal{M}$ and $PV=0$.
In particular, $V$ can be selected as a partial isometry such that $V^*V=\psi(I)$ when $\psi(I)$ is a projection.
\end{theorem}

\begin{proof}
Let $\mathcal{F}$ be a finite subset of $\mathcal{A}$ containing $I$, and $\varepsilon>0$.
Then there exists a sequence $\{\psi_n\}_{n\in\mathbb{N}}$ in $\mathfrak{F}=\mathfrak{F}(\mathcal{A},\mathcal{M},\mathcal{K}_{\mathcal{M}})$ such that
\begin{equation*}
  \Big\|\psi(A)-\sum_{n}\psi_n(A)\Big\|<\frac{\varepsilon}{2}~\text{for every}~A\in\mathcal{F}.
\end{equation*}
Since $\psi_n\in\mathfrak{F}$, we can write $\psi_n=\eta_n\circ\sigma_n$ for some completely positive maps $\sigma_n\colon\mathcal{A}\to\mathcal{B}_n$ and $\eta_n\colon\mathcal{B}_n\to\mathcal{M}$ with $\sigma_n|_{\mathcal{A}\cap\mathcal{K}_{\mathcal{M}}}=0$,
where $\mathcal{B}_n$ is a type $\mathrm{I}_{r(n)}$ unital subfactor of $\mathcal{M}$ with a system of matrix units $\{E_{st}^{(n)}\}_{1\leqslant s,t\leqslant r(n)}$.
It is clear that each $E_{ss}^{(n)}$ is an infinite projection in $\mathcal{M}$.

According to \cite[Proposition 2.1]{Haa}, there exists an operator $T_n$ in $\mathcal{M}$ such that $\eta_n(B)=T_n^*BT_n$ for every $B\in\mathcal{B}_n$.
By \Cref{Lemma 3.4}, there is an operator $V_n\in\mathcal{M}$ such that
\begin{equation}\label{equ sigma(A)-V*AV}
  r(n)^2\|T_n\|^2\cdot\|\sigma_n(A)-V_n^*AV_n\|<\frac{\varepsilon}{2^{n+2}}
  ~\text{for every}~A\in\mathcal{F},
\end{equation}
and $PV_n=0$.
For every $m,n\geqslant 0$, we define a finite subset of $\mathcal{M}$ by
\begin{equation*}
  \mathcal{F}_{m,n}=\{E_{1s}^{(m)}V_m^*AV_nE_{t1}^{(n)}
  \colon 1\leqslant s\leqslant r(m),1\leqslant t\leqslant r(n),A\in\mathcal{F}\}.
\end{equation*}
Based on \Cref{lem QmFQn=0}, we can find a sequence $\{Q_n\}_{n\in\mathbb{N}}$ of infinite projections in $\mathcal{M}$ such that $Q_n\leqslant E_{11}^{(n)}$ for each $n\geqslant 0$, and $Q_m\mathcal{F}_{m,n}Q_n=\{0\}$ when $m\ne n$.
Since $Q_n$ and $E_{11}^{(n)}$ are infinite projections, there exists a partial isometry $W_n$ in $\mathcal{M}$ such that
\begin{equation*}
  W_n^*W_n=E_{11}^{(n)},\quad W_nW_n^*=Q_n.
\end{equation*}
Since $Q_n\leqslant E_{11}^{(n)}$ and $E_{1s}^{(n)}\sigma_n(A)E_{t1}^{(n)}\in\mathbb{C}E_{11}^{(n)}$ for $1\leqslant s,t\leqslant r(n)$, it is straightforward to deduce that
\begin{equation*}
  E_{s1}^{(n)}W_n^*Q_nE_{1s}^{(n)}\sigma_n(A)E_{t1}^{(n)}Q_nW_nE_{1t}^{(n)}
  =E_{s1}^{(n)}E_{1s}^{(n)}\sigma_n(A)E_{t1}^{(n)}E_{1t}^{(n)}
  =E_{ss}^{(n)}\sigma_n(A)E_{tt}^{(n)}.
\end{equation*}
Consequently, $\sigma_n(A)=\sum_{s,t}
E_{s1}^{(n)}W_n^*Q_nE_{1s}^{(n)}\sigma_n(A)E_{t1}^{(n)}Q_nW_nE_{1t}^{(n)}$ and hence
\begin{equation}\label{equ sum-psi(A)}
  \sum_n\psi_n(A)=\sum_{n,s,t}
  T_n^*E_{s1}^{(n)}W_n^*Q_nE_{1s}^{(n)}\sigma_n(A)E_{t1}^{(n)}Q_nW_nE_{1t}^{(n)}T_n.
\end{equation}
For every $A\in\mathcal{F}$, since $E_{1s}^{(m)}V_m^*AV_nE_{t1}^{(n)}\in\mathcal{F}_{m,n}$, we have
\begin{equation*}
  Q_mE_{1s}^{(m)}V_m^*AV_nE_{t1}^{(n)}Q_n=0
\end{equation*}
when $m\ne n$.
Specifically, the operators $\{\sum_{t}V_nE_{t1}^{(n)}Q_nW_nE_{1t}^{(n)}T_n\}_{n\in\mathbb{N}}$ have orthogonal ranges when considering $A=I\in\mathcal{F}$.
Based on this, we can define an operator
\begin{equation*}
  V=\sum_{n,t}V_nE_{t1}^{(n)}Q_nW_nE_{1t}^{(n)}T_n.
\end{equation*}
Then
\begin{equation}\label{equ V*AV}
  V^*AV=\sum_{n,s,t}
  T_n^*E_{s1}^{(n)}W_n^*Q_nE_{1s}^{(n)}V_n^*AV_nE_{t1}^{(n)}Q_nW_nE_{1t}^{(n)}T_n
  ~\text{for every}~A\in\mathcal{F},
\end{equation}
and $PV=0$.
From \eqref{equ sigma(A)-V*AV}, \eqref{equ sum-psi(A)} and \eqref{equ V*AV}, it follows that
\begin{equation*}
  \Big\|\sum_n\psi_n(A)-V^*AV\Big\|<\frac{\varepsilon}{2}~\text{for every}~A\in\mathcal{F}.
\end{equation*}
Consequently, $\|\psi(A)-V^*AV\|<\varepsilon$ for every $A\in\mathcal{F}$.
In particular, $V^*V$ is a bounded operator if we take $A=I$.
Thus, we can conclude that $V$ belongs to $\mathcal{M}$.
Furthermore, due to $\|\psi(I)-V^*V\|<\varepsilon$, we can choose $V$ as a partial isometry such that $V^*V=\psi(I)$ when $\psi(I)$ is a projection.
\end{proof}

We now establish an enhanced version of our main theorem for separable unital $C^*$-subalgebras of semifinite factors.

\begin{theorem}\label{thm main-semifinite}
Let $\mathcal{M}$ be a separable properly infinite semifinite factor, $\mathcal{A}$ a separable unital $C^*$-subalgebra of $\mathcal{M}$, and $P\in\mathcal{K}_{\mathcal{M}}$ a finite projection.

Then for any $\psi\in\widehat{\mathfrak{SF}}$, there is a sequence $\{V_k\}_{k\in\mathbb{N}}$ in $\mathcal{M}$ such that
\begin{enumerate}
\item [(1)] $PV_k=0$ for every $k\in\mathbb{N}$.
\item [(2)] $\lim\limits_{k\to\infty}\|\psi(A)-V_k^*AV_k\|=0$ for every $A\in\mathcal{A}$.
\item [(3)] $\psi(A)-V_k^*AV_k\in\mathcal{K}_{\mathcal{M}}$ for every $A\in\mathcal{A}$ and $k\in\mathbb{N}$.
\end{enumerate}
In particular, $V_k$ can be selected as a partial isometry such that $V_k^*V_k=\psi(I)$ when $\psi(I)$ is a projection.
\end{theorem}

\begin{proof}
Let $\{Q_n\}_{n\in\mathbb{N}}$ be a sequence of finite projections in $\mathcal{K}_{\mathcal{M}}$ with $\bigvee_{n\in\mathbb{N}}Q_n=I$,
and $\mathcal{B}$ the separable unital $C^*$-subalgebra of $\mathcal{M}$ generated by $\psi(\mathcal{A})\cup\{Q_n\}_{n\in\mathbb{N}}$.
Then
\begin{equation*}
  \mathcal{I}=\{B\in\mathcal{B}\colon R(B)\in\mathcal{K}_{\mathcal{M}}\}
\end{equation*}
is an essential ideal of $\mathcal{B}$.
Additionally, let $\{A_j\}_{j\in\mathbb{N}}$ be a norm-dense sequence in $\mathcal{A}^{\mathrm{s.a.}}$ with $A_0=I$, where $\mathcal{A}^{\mathrm{s.a.}}$ is defined as $\{A\in\mathcal{A}\colon A^*=A\}$.

Fix $k\in\mathbb{N}$.
According to \cite[Theorem 2]{Arv}, there exists a sequence $\{E_n\}_{n\in\mathbb{N}}$ of positive operators in $\mathcal{I}$ such that $\sum_nE_n^2=I$, $\psi(A)-\sum_nE_n\psi(A)E_n\in\mathcal{K}_{\mathcal{M}}$ for every $A\in\mathcal{A}$, and
\begin{equation*}
  \Big\|\psi(A_j)-\sum_nE_n\psi(A_j)E_n\Big\|<\frac{1}{2^{k+1}}~\text{for all}~0\leqslant j\leqslant k.
\end{equation*}
We define $U_n$ inductively as follows.
For every $n\geqslant 0$, let
\begin{equation*}
  P_n=\bigvee\{P,R(A_jU_mE_m)\colon 0\leqslant j\leqslant n+k,0\leqslant m\leqslant n-1\}.
\end{equation*}
By definition, we have $P_0=P$.
Since $P$ and $R(E_m)$ are finite, $P_n$ is also finite.
By \Cref{thm main}, there exists an operator $U_n$ in $\mathcal{M}$ such that
\begin{equation}\label{equ psi(A)-U*AU}
  \|\psi(A_j)-U_n^*A_jU_n\|<\frac{1}{2^{n+k+2}}~\text{for all}~0\leqslant j\leqslant n+k,
\end{equation}
and $P_nU_n=0$.
For every $0\leqslant j\leqslant n+k$ and $0\leqslant m\leqslant n-1$, by the definition of $P_n$, we have $P_nA_jU_mE_m=A_jU_mE_m$ and hence $E_mU_m^*A_jU_nE_n=E_mU_m^*A_jP_nU_nE_n=0$.
Since each $A_j$ is selfadjoint, it follows that
\begin{equation}\label{equ finite-sum}
  E_mU_m^*A_jU_nE_n=0~\text{whenever}~0\leqslant j\leqslant\max\{m,n\}+k, m\ne n.
\end{equation}
Specifically, the operators $\{U_nE_n\}_{n\in\mathbb{N}}$ have orthogonal ranges when considering $A_0=I$.
Based on this, we can define an operator
\begin{equation*}
  V=\sum_nU_nE_n.
\end{equation*}
Then
\begin{equation}\label{equ 2-terms}
  \sum_nE_n\psi(A_j)E_n-V^*A_jV
  =\sum_nE_n\big(\psi(A_j)-U_n^*A_jU_n\big)E_n
  -\sum_{m\ne n}E_mU_m^*A_jU_nE_n
\end{equation}
for every $j\geqslant 0$, and $PV=0$.
On the right-hand side of \eqref{equ 2-terms}, the first term is norm-convergent by \eqref{equ psi(A)-U*AU}, and the second term is a finite sum by \eqref{equ finite-sum}.
Since each summand lies in $\mathcal{K}_{\mathcal{M}}$, it follows that
\begin{equation*}
  \sum_nE_n\psi(A_j)E_n-V^*A_jV\in\mathcal{K}_{\mathcal{M}}~\text{for all}~j\geqslant 0.
\end{equation*}
We further have the estimation
\begin{equation*}
  \Big\|\sum_nE_n\psi(A_j)E_n-V^*A_jV\Big\|<\frac{1}{2^{k+1}}~\text{for all}~0\leqslant j\leqslant k.
\end{equation*}
Therefore, $\psi(A_j)-V^*A_jV\in\mathcal{K}_{\mathcal{M}}$ for all $j\geqslant 0$, and $\|\psi(A_j)-V^*A_jV\|<2^{-k}$ for all $0\leqslant j\leqslant k$.
In particular, $V^*V$ is a bounded operator if we consider $A_0=I$.
We can conclude that $V$ belongs to $\mathcal{M}$.
Now we set $V_k=V$.
\end{proof}

\section{Generalized Voiculescu's Theorem}\label{sec Voi}

In this section, we focus on unital *-homomorphisms in $\widehat{\mathfrak{SF}}
=\widehat{\mathfrak{SF}}(\mathcal{A},\mathcal{M},\mathcal{K}_{\mathcal{M}})$ as defined in \Cref{def SF}.

\begin{lemma}\label{lem VA-AV}
Let $\mathcal{M}$ be a separable properly infinite factor, and $\mathcal{A}$ a separable unital $C^*$-subalgebra of $\mathcal{M}$.

If $\varphi\in\widehat{\mathfrak{SF}}$ is a unital *-homomorphism, then there is a sequence $\{V_k\}_{k\in\mathbb{N}}$ of isometries in $\mathcal{M}$ such that
\begin{equation*}
  \lim_{k\to\infty}\|V_k\varphi(A)-AV_k\|=0~\text{for every}~A\in\mathcal{A}.
\end{equation*}
Furthermore, if $\mathcal{M}$ is semifinite, we can choose $\{V_k\}_{k\in\mathbb{N}}$ such that
\begin{equation*}
  V_k\varphi(A)-AV_k\in\mathcal{K}_{\mathcal{M}}~\text{for every}~A\in\mathcal{A}~\text{and}~k\in\mathbb{N}.
\end{equation*}
\end{lemma}

\begin{proof}
By \Cref{thm main}, there exists a sequence $\{V_k\}_{k\in\mathbb{N}}$ of isometries in $\mathcal{M}$ such that
\begin{equation*}
  \lim_{k\to\infty}\|\varphi(A)-V_k^*AV_k\|=0~\text{for every}~A\in\mathcal{A}.
\end{equation*}
Since $\varphi$ is a unital *-homomorphism, we have
\begin{equation}\label{equ VA-AV}
\begin{split}
    &\big(V_k\varphi(A)-AV_k\big)^*\big(V_k\varphi(A)-AV_k\big)\\
    =\ &\varphi(A^*)\big(\varphi(A)-V_k^*AV_k\big)
    +\big(\varphi(A^*)-V_k^*A^*V_k\big)\varphi(A)
    -\big(\varphi(A^*A)-V_k^*A^*AV_k\big).
\end{split}
\end{equation}
It follows that $\lim\limits_{k\to\infty}\|V_k\varphi(A)-AV_k\|=0$ for every $A\in\mathcal{A}$.
Furthermore, if $\mathcal{M}$ is semifinite, then by \Cref{thm main-semifinite}, we can assume that $\varphi(A)-V_k^*AV_k\in\mathcal{K}_{\mathcal{M}}$ for every $A\in\mathcal{A}$ and $k\in\mathbb{N}$.
As a result, we can deduce $V_k\varphi(A)-AV_k\in\mathcal{K}_{\mathcal{M}}$ from \eqref{equ VA-AV}.
\end{proof}

The following theorem is known as Voiculescu's theorem \cite[Theorem 1.3]{Voi} when $\mathcal{M}$ is a separable type $\mathrm{I}_\infty$ factor.
We will employ the notation $P^{\perp}=I-P$ for a projection $P$ in $\mathcal{M}$.

\begin{theorem}\label{thm Voi}
Let $\mathcal{M}$ be a separable properly infinite factor, and $\mathcal{A}$ a separable unital $C^*$-subalgebra of $\mathcal{M}$.

If $\varphi\in\widehat{\mathfrak{SF}}$ is a unital *-homomorphism, then there is a sequence $\{V_k\}_{k\in\mathbb{N}}$ of isometries in $\mathcal{M}\otimes M_2(\mathbb{C})$ such that
\begin{equation*}
  \lim_{k\to\infty}\|(A\oplus\varphi(A))-V_k^*(A\oplus 0)V_k\|=0~\text{for every}~A\in\mathcal{A},
\end{equation*}
and $V_k^*V_k=I\oplus I,V_kV_k^*=I\oplus 0$ for every $k\in\mathbb{N}$.
Furthermore, if $\mathcal{M}$ is semifinite, we can choose $\{V_k\}_{k\in\mathbb{N}}$ such that
\begin{equation*}
  (A\oplus\varphi(A))-V_k^*(A\oplus 0)V_k\in\mathcal{K}_{\mathcal{M}}\otimes M_2(\mathbb{C})~\text{for every}~A\in\mathcal{A}~\text{and}~k\in\mathbb{N}.
\end{equation*}
\end{theorem}

\begin{proof}
Let $\{E_{mn}\}_{m,n\in\mathbb{N}}$ be a system of matrix units in $\mathcal{M}$ such that $\sum_nE_{nn}=I$ and $E_{00}$ is an infinite projection in $\mathcal{M}$.
Let $T$ be an isometry in $\mathcal{M}$ with $T^*T=I$ and $TT^*=E_{00}$, and let $S$ denote the isometry $\sum_nE_{n+1,n}$ in $\mathcal{M}$.
We define a map
\begin{equation*}
  \psi\colon\mathcal{A}\to\mathcal{M},\quad A\mapsto\sum_nE_{n0}T\varphi(A)T^*E_{0n}.
\end{equation*}
Clearly, $\psi$ is a unital *-homomorphism and lies in $\widehat{\mathfrak{SF}}$ by \Cref{lem SF-addition}.
By \Cref{lem VA-AV}, we can find a sequence $\{U_k\}_{k\in\mathbb{N}}$ of isometries in $\mathcal{M}$ such that
\begin{equation}\label{equ UA-AU}
  \lim_{k\to\infty}\|U_k\psi(A)-AU_k\|=0~\text{for every}~A\in\mathcal{A}.
\end{equation}
Furthermore, if $\mathcal{M}$ is semifinite, then we can assume that
\begin{equation*}
  U_k\varphi(A)-AU_k\in\mathcal{K}_{\mathcal{M}}~\text{for every}~A\in\mathcal{A}~\text{and}~k\in\mathbb{N}.
\end{equation*}
Let $P_k$ be the projection $U_kU_k^*$ in $\mathcal{M}$, $W_k$ the isometry $TU_kT^*+I-E_{00}$ in $\mathcal{M}$, and $F_k$ the unitary operator $P_k^{\perp}T^*+U_kW_k^*$ in $\mathcal{M}$.
Then $W_k^*W_k=I$, $W_kW_k^*=I-TP_k^{\perp}T^*$, and
\begin{equation*}
  F_k^*AF_k
  =TP_k^{\perp}AP_k^{\perp}T^*+TP_k^{\perp}AU_kW_k^*
  +W_kU_k^*AP_k^{\perp}T^*+W_kU_k^*AU_kW_k^*.
\end{equation*}
Since $P_k^{\perp}AU_k=\big(AU_k-U_k\psi(A)\big)+U_k\big(\psi(A)U_k^*-U_k^*A\big)U_k$, we deduce from \eqref{equ UA-AU} that
\begin{equation*}
  \lim_{k\to\infty}\|P_k^{\perp}AU_k\|=0~\text{for every}~A\in\mathcal{A}.
\end{equation*}
It follows that
\begin{equation}\label{equ FAF}
  \lim_{k\to\infty}\|F_k^*AF_k-(TP_k^{\perp}AP_k^{\perp}T^*+W_k\psi(A)W_k^*)\|=0
  ~\text{for every}~A\in\mathcal{A}.
\end{equation}
Let
\begin{equation*}
  X_k=\begin{pmatrix}
  TP_k^{\perp}T^*+W_kSW_k^* & W_kT\\
  0 & 0
\end{pmatrix}\in\mathcal{M}\otimes M_2(\mathbb{C}).
\end{equation*}
Then $X_k^*X_k=I\oplus I$ and $X_kX_k^*=I\oplus 0$.
Since $S^*\psi(A)S=\psi(A)$ and $T^*\psi(A)T=\varphi(A)$, we have
\begin{equation*}
  X_k^*
  \begin{pmatrix}
    TP_k^{\perp}AP_k^{\perp}T^*+W_k\psi(A)W_k^* & 0 \\
    0 & 0
  \end{pmatrix}X_k=
  \begin{pmatrix}
    TP_k^{\perp}AP_k^{\perp}T^*+W_k\psi(A)W_k^* & 0 \\
    0 & \varphi(A)
  \end{pmatrix}.
\end{equation*}
Hence \eqref{equ FAF} implies that
\begin{equation*}
  \lim\limits_{k\to\infty}\|X_k^*(F_k^*AF_k\oplus 0)X_k
  -(F_k^*AF_k\oplus\varphi(A))\|=0~\text{for every}~A\in\mathcal{A}.
\end{equation*}
Now we set $V_k=(F_k\oplus I)X_k(F_k^*\oplus I)$.
\end{proof}

According to \Cref{prop into-typeI}, the following theorem is a special case of \Cref{thm Voi}.

\begin{theorem}\label{thm Voi-into-typeI}
Let $\mathcal{M}$ be a separable properly infinite factor, $\mathcal{A}$ a separable unital $C^*$-subalgebra of $\mathcal{M}$, and $\mathcal{B}$ a type $\mathrm{I}_{\infty}$ unital subfactor of $\mathcal{M}$.

Then for any unital *-homomorphism $\varphi\colon\mathcal{A}\to\mathcal{B}$ with $\varphi|_{\mathcal{A}\cap\mathcal{K}_{\mathcal{M}}}=0$, there exists a sequence $\{V_k\}_{k\in\mathbb{N}}$ of isometries in $\mathcal{M}\otimes M_2(\mathbb{C})$ such that
\begin{equation*}
  \lim_{k\to\infty}\|(A\oplus\varphi(A))-V_k^*(A\oplus 0)V_k\|=0~\text{for every}~A\in\mathcal{A},
\end{equation*}
and $V_k^*V_k=I\oplus I,V_kV_k^*=I\oplus 0$ for every $k\in\mathbb{N}$.
Furthermore, if $\mathcal{M}$ is semifinite, we can choose $\{V_k\}_{k\in\mathbb{N}}$ such that
\begin{equation*}
  (A\oplus\varphi(A))-V_k^*(A\oplus 0)V_k\in\mathcal{K}_{\mathcal{M}}\otimes M_2(\mathbb{C})~\text{for every}~A\in\mathcal{A}~\text{and}~k\in\mathbb{N}.
\end{equation*}
\end{theorem}

\section{Applications}\label{sec application}

We provide several applications of generalized Voiculescu's theorem in this section.

\subsection{Reducible Operators}\label{subsec reducible}

Let $\mathcal{M}$ be a separable properly infinite factor, and $T$ an operator in $\mathcal{M}$.
We say that $T$ is \emph{reducible} in $\mathcal{M}$ if there is a projection $P$ in $\mathcal{M}$ such that $PT=TP$ and $P\ne 0,I$.

\begin{theorem}\label{thm reducible}
Let $\mathcal{M}$ be a separable properly infinite factor.
Then the set of all reducible operators is norm-dense in $\mathcal{M}$.
\end{theorem}

\begin{proof}
Let $\mathcal{B}$ be a type $\mathrm{I}_{\infty}$ unital subfactor of $\mathcal{M}$, and $T\in\mathcal{M}$.
Let $\mathcal{A}$ be the separable unital $C^*$-algebra generated by $T$, and $\mathcal{I}=\mathcal{A}\cap\mathcal{K}_{\mathcal{M}}$.

Let $\psi\colon\mathcal{A}/\mathcal{I}\to\mathcal{B}$ be a unital *-homomorphism,
$\pi_1\colon\mathcal{A}\to\mathcal{A}/\mathcal{I}$ the quotient map,
and $\varphi=\psi\circ\pi_1\colon\mathcal{A}\to\mathcal{B}$.
By \Cref{thm Voi-into-typeI}, there is a sequence $\{V_k\}_{k\in\mathbb{N}}$ of isometries in $\mathcal{M}\otimes M_2(\mathbb{C})$ such that
\begin{equation*}
  \lim_{k\to\infty}\|(A\oplus\varphi(A))-V_k^*(A\oplus 0)V_k\|=0~\text{for every}~A\in\mathcal{A},
\end{equation*}
and $V_k^*V_k=I\oplus I,V_kV_k^*=I\oplus 0$ for every $k\in\mathbb{N}$.
We can write
\begin{equation*}
  V_k(T\oplus\varphi(T))V_k^*=T_k\oplus 0\quad\text{and}\quad V_k(I\oplus 0)V_k^*=P_k\oplus 0.
\end{equation*}
It is clear that $P_kT_k=T_kP_k$ and $P_k\ne 0,I$.
Therefore, $T_k$ is reducible in $\mathcal{M}$.
Moreover, we have $\lim\limits_{k\to\infty}\|T_k-T\|=0$.
This completes the proof.
\end{proof}

\subsection{Voiculescu’s Bicommutant Theorem}\label{sec bicommutant}

Let $\mathcal{M}$ be a separable properly infinite semifinite factor, and $\mathscr{A}$ a unital subalgebra of $\mathcal{M}/\mathcal{K}_{\mathcal{M}}$.
As defined in \cite[Page 344]{Arv}, the \emph{essential lattice} $Lat_e(\mathscr{A})$ of $\mathscr{A}$ is the set of all projections $p$ in $\mathcal{M}/\mathcal{K}_{\mathcal{M}}$ such that $p^{\perp}ap=0$ for every $a\in\mathscr{A}$.
If $t\in\mathcal{M}/\mathcal{K}_{\mathcal{M}}$, then $\|p^{\perp}tp\|=\|p^{\perp}(t-a)p\|\leqslant\|t-a\|$ for every $a\in\mathscr{A}$.
It follows that
\begin{equation*}
  \sup_{p}\|p^{\perp}tp\|\leqslant\mathrm{dist}(t,\mathscr{A}),
\end{equation*}
where $\mathrm{dist}(t,\mathscr{A})=\inf\{\|t-a\|\colon a\in\mathscr{A}\}$.
The subsequent result is commonly referred to as Arveson's distance formula.

\begin{lemma}\label{lem distance}
Let $\mathcal{M}$ be a separable properly infinite semifinite factor, and $\mathscr{A}$ a separable unital subalgebra of $\mathcal{M}/\mathcal{K}_{\mathcal{M}}$.

Then for any $t$ in $\mathcal{M}/\mathcal{K}_{\mathcal{M}}$, there is a projection $q$ in $Lat_e(\mathscr{A})$ such that
\begin{equation*}
  \|q^{\perp}tq\|=\mathrm{dist}(t,\mathscr{A}).
\end{equation*}
\end{lemma}

\begin{proof}
Recall that $\pi\colon\mathcal{M}\to\mathcal{M}/\mathcal{K}_{\mathcal{M}}$ is the canonical quotient map.
Let $\mathscr{A}_t$ be the separable unital $C^*$-algebra generated by $t$ and $\mathscr{A}$, and $\mathcal{B}$ a type $\mathrm{I}_{\infty}$ unital subfactor of $\mathcal{M}$.
By the GNS construction, there is a unital *-homomorphism $\sigma\colon\mathscr{A}_t\to\mathcal{B}$ and a $\sigma(\mathscr{A})$-invariant projection $P$ in $\mathcal{B}$ such that
\begin{equation*}
  \|P^{\perp}\sigma(t)P\|_e\geqslant\mathrm{dist}(t,\mathscr{A}),
\end{equation*}
where $\|A\|_e=\|\pi(A)\|$ for every $A\in\mathcal{M}$.

Let $\mathcal{A}_t$ be a separable unital $C^*$-subalgebra of $\mathcal{M}$ such that $\pi(\mathcal{A}_t)=\mathscr{A}_t$, and
\begin{equation*}
  \varphi=\sigma\circ\pi\colon\mathcal{A}_t\to\mathcal{B}.
\end{equation*}
Then $\varphi$ is a unital *-homomorphism with $\varphi|_{\mathcal{A}_t\cap\mathcal{K}_{\mathcal{M}}}=0$.
By \Cref{thm Voi-into-typeI}, there is an isometry $V$ in $\mathcal{M}\otimes M_2(\mathbb{C})$ such that
\begin{equation*}
  (A\oplus\varphi(A))-V^*(A\oplus0)V\in\mathcal{K}_{\mathcal{M}}\otimes M_2(\mathbb{C})\text{ for every } A\in\mathcal{A}_t,
\end{equation*}
and $V^*V=I\oplus I,VV^*=I\oplus 0$.

Let $\mathcal{A}=\{A\in\mathcal{A}_t\colon\pi(A)\in\mathscr{A}\}$ and $Q\oplus 0=V(0\oplus P)V^*$.
Since $\pi(\mathcal{A})=\mathscr{A}$ and $\varphi(\mathcal{A})=\sigma(\mathscr{A})$, the projection $P$ is $\varphi(\mathcal{A})$-invariant.
We conclude that $Q^{\perp}AQ\in\mathcal{K}_{\mathcal{M}}$ for every $A\in\mathcal{A}$.
This implies that $q=\pi(Q)$ belongs to $Lat_e(\mathscr{A})$.
Choose an operator $T$ in $\mathcal{A}_t$ such that $\pi(T)=t$.
Then
\begin{equation*}
  (Q^{\perp}TQ\oplus 0)-V(0\oplus P^{\perp}\varphi(T)P)V^*\in\mathcal{K}_{\mathcal{M}}\otimes M_2(\mathbb{C}).
\end{equation*}
It follows that $\|q^{\perp}tq\|=\|Q^{\perp}TQ\|_e=\|P^{\perp}\varphi(T)P\|_e
=\|P^{\perp}\sigma(t)P\|_e\geqslant\mathrm{dist}(t,\mathscr{A})$.
\end{proof}

Recall that the \emph{relative commutant} of a unital $C^*$-subalgebra $\mathscr{A}$ of $\mathcal{M}/\mathcal{K}_{\mathcal{M}}$ is defined as
\begin{equation*}
  \mathscr{A}^c=\{t\in\mathcal{M}/\mathcal{K}_{\mathcal{M}}\colon
  ta=at\text{ for all } a\in\mathscr{A}\}.
\end{equation*}
The \emph{relative bicommutant} of $\mathscr{A}$ is $\mathscr{A}^{cc}=(\mathscr{A}^c)^c$.
It follows from \Cref{lem distance} that every separable norm-closed unital subalgebra of $\mathcal{M}/\mathcal{K}_{\mathcal{M}}$ is reflexive.
In particular, the following generalization of Voiculescu's relative bicommutant theorem holds.

\begin{theorem}\label{thm bicommutant}
Let $\mathcal{M}$ be a separable properly infinite semifinite factor.
Then every separable unital $C^*$-subalgebra of $\mathcal{M}/\mathcal{K}_{\mathcal{M}}$ equals its relative bicommutant.
\end{theorem}

Let $\mathcal{M}$ be a separable type $\mathrm{III}$ factor.
It is worth noting that $\mathcal{K}_{\mathcal{M}}=\{0\}$ and hence we have $\mathcal{M}/\mathcal{K}_{\mathcal{M}}=\mathcal{M}$.
Let $\mathcal{B}$ be a type $\mathrm{I}_\infty$ unital subfactor of $\mathcal{M}$, and $\mathcal{A}$ the separable unital $C^*$-subalgebra $\mathbb{C}I+\mathcal{K}_{\mathcal{B}}$ of $\mathcal{B}$.
Then the relative bicommutant of $\mathcal{A}$ in $\mathcal{M}$ is equal to $\mathcal{B}$.
From this, a version of \Cref{thm bicommutant} does not hold for type $\mathrm{III}$ factors.
In the next subsection, we will present a kind of asymptotic bicommutant theorem.

\subsection{Asymptotic bicommutant theorem}\label{subsec asy-bicommutant}

Let $\mathcal{M}$ be a separable properly infinite factor, and $\mathcal{A}$ a unital subalgebra of $\mathcal{M}$.

\begin{definition}\label{def appr-bicommutant}
The \emph{relative approximate bicommutant} $\mathrm{appr}(\mathcal{A})^{cc}$ of $\mathcal{A}$ in $\mathcal{M}$ is defined as the set of all operators $T$ in $\mathcal{M}$ such that $\|P_nT-TP_n\|\to 0$ whenever $\{P_n\}_{n\in\mathbb{N}}$ is a sequence of projections in $\mathcal{M}$ such that $\|P_nA-AP_n\|\to 0$ for every $A\in\mathcal{A}$.
\end{definition}

The following theorem is a generalization of Hadwin's asymptotic bicommutant theorem \cite{Had} in type $\mathrm{III}$ factors.

\begin{theorem}\label{thm asy-bicommutant}
Let $\mathcal{M}$ be a separable type $\mathrm{III}$ factor.
Then every separable unital $C^*$-subalgebra of $\mathcal{M}$ is equal to its relative approximate bicommutant.
\end{theorem}

\Cref{thm asy-bicommutant} is a consequence of the following asymptotic distance formula, whose proof follows a similar argument as in \Cref{lem distance}.

\begin{lemma}\label{lem asy-distance}
Let $\mathcal{M}$ be a separable type $\mathrm{III}$ factor, and $\mathcal{A}$ a separable unital subalgebra of $\mathcal{M}$.
Then for any operator $T$ in $\mathcal{M}$, there exists a sequence $\{Q_k\}_{k\in\mathbb{N}}$ of projections in $\mathcal{M}$ such that
\begin{equation*}
  \lim\limits_{k\to\infty}\|Q_k^{\perp}AQ_k\|=0~\text{for every}~A\in\mathcal{A},
\end{equation*}
and
\begin{equation*}
  \lim\limits_{k\to\infty}\|Q_k^{\perp}TQ_k\|=\mathrm{dist}(T,\mathcal{A}).
\end{equation*}
\end{lemma}

\begin{proof}
Let $\mathcal{A}_T$ be the separable unital $C^*$-subalgebra of $\mathcal{M}$ generated by $T$ and $\mathcal{A}$, and $\mathcal{B}$ a type $\mathrm{I}_{\infty}$ unital subfactor of $\mathcal{M}$.
Then there exists a unital *-homomorphism $\varphi\colon\mathcal{A}_T\to\mathcal{B}$ and a $\varphi(\mathcal{A})$-invariant projection $P$ in $\mathcal{B}$ such that
\begin{equation*}
  \|P^{\perp}\varphi(T)P\|\geqslant\mathrm{dist}(T,\mathcal{A}).
\end{equation*}
Note that $\mathcal{K}_{\mathcal{M}}=\{0\}$ and the condition $\varphi|_{\mathcal{A}_T\cap\mathcal{K}_{\mathcal{M}}}=0$ holds evidently.
By \Cref{thm Voi-into-typeI}, there exists a sequence $\{V_k\}_{k\in\mathbb{N}}$ of isometries in $\mathcal{M}\otimes M_2(\mathbb{C})$ such that
\begin{equation*}
  \lim_{k\to\infty}\|(A\oplus\varphi(A))-V_k^*(A\oplus 0)V_k\|=0~\text{for every}~A\in\mathcal{A}_T,
\end{equation*}
and $V_k^*V_k=I\oplus I,V_kV_k^*=I\oplus 0$ for every $k\in\mathbb{N}$.
Let
\begin{equation*}
  Q_k\oplus 0=V_k(0\oplus P)V_k^*.
\end{equation*}
Then the sequence $\{Q_k\}_{k\in\mathbb{N}}$ has the desired property.
\end{proof}

\subsection{The First Cohomology Group}\label{subsec cohomology}

Let $\mathcal{M}$ be a separable properly infinite semifinite factor, and $\mathcal{A}$ a unital $C^*$-subalgebra of $\mathcal{M}$.

\begin{definition}\label{def cohomology}
A linear map $\delta\colon\mathcal{A}\to\mathcal{K}_{\mathcal{M}}$ is said to be a \emph{derivation} if it satisfies the Leibniz rule
\begin{equation*}
  \delta(AB)=\delta(A)B+A\delta(B).
\end{equation*}
The set of all derivations of $\mathcal{A}$ into $\mathcal{K}_{\mathcal{M}}$ is denoted by $\mathrm{Der}(\mathcal{A},\mathcal{K}_{\mathcal{M}})$.

For any operator $K$ in $\mathcal{K}_{\mathcal{M}}$, the \emph{inner derivation} $\delta_K\colon\mathcal{A}\to\mathcal{K}_{\mathcal{M}}$ is given by
\begin{equation*}
  \delta_K(A)=KA-AK.
\end{equation*}
The set of all inner derivations of $\mathcal{A}$ into $\mathcal{K}_{\mathcal{M}}$ is denoted by $\mathrm{Inn}(\mathcal{A},\mathcal{K}_{\mathcal{M}})$.

The quotient space $H^1(\mathcal{A},\mathcal{K}_{\mathcal{M}})
=\mathrm{Der}(\mathcal{A},\mathcal{K}_{\mathcal{M}})
/\mathrm{Inn}(\mathcal{A},\mathcal{K}_{\mathcal{M}})$ is called the \emph{first cohomology group} of $\mathcal{A}$ with coefficients in $\mathcal{K}_{\mathcal{M}}$.
\end{definition}

Since $\mathrm{Inn}(\mathcal{A},\mathcal{K}_{\mathcal{M}})$ is a linear subspace of $\mathrm{Der}(\mathcal{A},\mathcal{K}_{\mathcal{M}})$, the first cohomology group $H^1(\mathcal{A},\mathcal{K}_{\mathcal{M}})$ is also a linear space.
We do not require a topological structure in \Cref{def cohomology}.

Now we introduce some notation.
If $\mathcal{A}$ is a unital $C^*$-subalgebra of $\mathcal{B}(\mathcal{H})$, then its \emph{commutant} $\mathcal{A}'$ is the set of all operators in $\mathcal{B}(\mathcal{H})$ commuting with all operators in $\mathcal{A}$.
The von Neumann bicommutant theorem asserts that the bicommutant $\mathcal{A}''$ is the von Neumann algebra generated by $\mathcal{A}$.

If $\mathcal{A}$ is a unital $C^*$-subalgebra of $\mathcal{M}$, then the \emph{relative commutant} of $\mathcal{A}$ in $\mathcal{M}$ is denoted by
\begin{equation*}
  \mathcal{A}^c=\{T\in\mathcal{M}\colon TA=AT~\text{for all}~A\in\mathcal{A}\}.
\end{equation*}
Since $\mathcal{A}^c=\mathcal{A}'\cap\mathcal{M}\subseteq\mathcal{A}'$, we have $\mathcal{A}^{cc}=(\mathcal{A}^c)'\cap\mathcal{M}
\supseteq(\mathcal{A}')'\cap\mathcal{M}=\mathcal{A}''\supseteq\mathcal{A}$.
Hence the relative bicommutant $\mathcal{A}^{cc}$ contains $\mathcal{A}$.
Similarly, the \emph{relative commutant} of a unital $C^*$-subalgebra $\mathscr{A}$ of $\mathcal{M}/\mathcal{K}_{\mathcal{M}}$ is denoted by
\begin{equation*}
  \mathscr{A}^c=\{t\in\mathcal{M}/\mathcal{K}_{\mathcal{M}}\colon
  ta=at\text{ for all } a\in\mathscr{A}\}.
\end{equation*}
It is clear that $\pi(\mathcal{A})^c\supseteq\pi(\mathcal{A}^c)$, where $\pi\colon\mathcal{M}\to\mathcal{M}/\mathcal{K}_{\mathcal{M}}$ is the canonical quotient map.

The following theorem is similar to \cite[Theorem 2.2]{PR}, which states that not all derivations of $\mathcal{A}$ into $\mathcal{K}_{\mathcal{M}}$ are inner under certain conditions.

\begin{theorem}\label{thm cohomology}
Let $\mathcal{M}$ be a separable properly infinite semifinite factor, and $\mathcal{A}$ a separable unital $C^*$-subalgebra of $\mathcal{M}$.

If $\pi(\mathcal{A}'')$ is infinite-dimensional, then $H^1(\mathcal{A},\mathcal{K}_{\mathcal{M}})\ne\{0\}$.
\end{theorem}

\begin{proof}
If $\pi(T)\in\pi(\mathcal{A})^c$, then $\delta_T(A)=TA-AT$ maps $\mathcal{A}$ into $\mathcal{K}_{\mathcal{M}}$, and is clearly a derivation in $\mathrm{Der}(\mathcal{A},\mathcal{K}_{\mathcal{M}})$.
Moreover, if $\pi(T)=\pi(S)$, then $T-S\in\mathcal{K}_{\mathcal{M}}$.
It follows that $\delta_T-\delta_S=\delta_{T-S}
\in\mathrm{Inn}(\mathcal{A},\mathcal{K}_{\mathcal{M}})$.
Thus, we have a well-defined linear map
\begin{equation*}
  \varphi\colon\pi(\mathcal{A})^c\to H^1(\mathcal{A},\mathcal{K}_{\mathcal{M}}),
  \quad \pi(T)\mapsto\delta_T+\mathrm{Inn}(\mathcal{A},\mathcal{K}_{\mathcal{M}}).
\end{equation*}
If $\pi(T)\in\ker\varphi$, then there exists an operator $K\in\mathcal{K}_{\mathcal{M}}$ such that $\delta_T=\delta_K$.
It follows that $T-K\in\mathcal{A}^c$, and hence $\pi(T)\in\pi(\mathcal{A}^c)$.
Therefore, the induced map
\begin{equation*}
  \widetilde{\varphi}\colon\pi(\mathcal{A})^c/\pi(\mathcal{A}^c)\to H^1(\mathcal{A},\mathcal{K}_{\mathcal{M}})
\end{equation*}
is injective.
It suffices to show that $\pi(\mathcal{A})^c\ne\pi(\mathcal{A}^c)$.

Suppose on the contrary, that $\pi(\mathcal{A})^c=\pi(\mathcal{A}^c)$.
Since $\pi(\mathcal{A})$ is a separable unital $C^*$-subalgebra of $\mathcal{M}/\mathcal{K}_{\mathcal{M}}$, we have $\pi(\mathcal{A})=\pi(\mathcal{A})^{cc}$ by \Cref{thm bicommutant}.
It follows that
\begin{equation*}
  \pi(\mathcal{A}^{cc})\supseteq\pi(\mathcal{A}'')\supseteq\pi(\mathcal{A})
  =\pi(\mathcal{A})^{cc}=\pi(\mathcal{A}^c)^c\supseteq\pi(\mathcal{A}^{cc}).
\end{equation*}
Therefore, we obtain $\pi(\mathcal{A}'')=\pi(\mathcal{A})$, which is an infinite-dimensional separable $C^*$-algebra.
This contradicts the next result, \Cref{prop pi(N)}.
\end{proof}

\begin{remark}
Let $\mathcal{A}$ be a separable infinite-dimensional unital $C^*$-subalgebra of $\mathcal{M}$.
Then $\mathcal{A}''$ is an infinite-dimensional von Neumann subalgebra of $\mathcal{M}$.
If $\mathcal{M}$ is a factor of type $\mathrm{I}_\infty$, then it is not hard to see that $\pi(\mathcal{A}'')$ is also infinite-dimensional.
Therefore, \Cref{thm cohomology} is a generalization of \cite[Theorem 2.2]{PR}.

If $\mathcal{M}$ is a factor of type $\mathrm{II}_\infty$, then it is possible that $\pi(\mathcal{A}'')$ is finite-dimensional.
For example, let $P$ be a nonzero finite projection in $\mathcal{M}$, $\mathcal{A}_0$ a separable infinite-dimensional $C^*$-subalgebra of the type $\mathrm{II}_1$ factor $P\mathcal{M}P$ such that $P\in\mathcal{A}_0$, and $\mathcal{A}=\mathcal{A}_0+\mathbb{C}(I-P)$.
Then $\mathcal{A}''\subseteq P\mathcal{M}P+\mathbb{C}(I-P)$ and hence $\pi(\mathcal{A}'')=\mathbb{C}\pi(I)$.
\end{remark}

\begin{example}
We provide the following two examples.
\begin{enumerate}
\item [(1)] Let $\mathcal{M}$ be a separable type $\mathrm{I}_\infty$ factor and $\mathcal{A}=\mathbb{C}I+\mathcal{K}_{\mathcal{M}}$.
    Then $\mathcal{A}''=\mathcal{M}$ and it follows that $H^1(\mathcal{A},\mathcal{K}_{\mathcal{M}})\ne\{0\}$.

\item [(2)] Let $\mathcal{M}$ be the type $\mathrm{II}_{\infty}$ factor $\mathcal{N}\,\overline{\otimes}\,\mathcal{B}(L^2(\mathbb{T},\mu))$, where $\mathcal{N}$ is a separable type $\mathrm{II}_1$ factor and $\mu$ is the Haar measure on the unit circle $\mathbb{T}$.
    Suppose that $C(\mathbb{T})$ acts on $L^2(\mathbb{T},\mu)$ by multiplication.
    If $\mathcal{A}=I_{\mathcal{N}}\otimes C(\mathbb{T})$, then $H^1(\mathcal{A},\mathcal{K}_{\mathcal{M}})\ne\{0\}$.
\end{enumerate}
\end{example}

Although the following proposition is well-known to experts, we include its proof for completeness.

\begin{proposition}\label{prop pi(N)}
Let $\mathcal{M}$ be a separable properly infinite semifinite factor, and $\mathcal{N}$ a unital von Neumann subalgebra of $\mathcal{M}$.
Then the $C^*$-algebra $\pi(\mathcal{N})$ is either finite-dimensional or non-separable.
\end{proposition}

\begin{proof}
Suppose that $\pi(\mathcal{N})$ is an infinite-dimensional $C^*$-algebra.
According to \cite[Exercise 4.6.13]{KR}, there exists a selfadjoint element $A$ in $\mathcal{N}$ such that $\pi(A)$ has infinite spectrum.
We can find a sequence $\{[a_n,b_n]\}_{n\in\mathbb{N}}$ of disjoint intervals such that each interval contains a spectral point of $\pi(A)$.
Let $f_n$ be a continuous function on $\mathbb{R}$, which is positive within the interval $(a_n,b_n)$, and zero elsewhere.
Then $f_n(\pi(A))\ne 0$.

Let $\chi_n$ be the characteristic function of the interval $[a_n,b_n]$, and $P_n$ the spectral projection $\chi_n(A)$ for every $n\in\mathbb{N}$.
Let $\mathcal{A}$ be the set of all operators of the form $\sum_{n\in\mathbb{N}}c_nP_n$, where $\{c_n\}_{n\in\mathbb{N}}$ is a bounded complex sequence in $\ell^\infty$.
Clearly, $\mathcal{A}$ is *-isomorphic to $\ell^\infty$ and is a subset of $\mathcal{N}$.
For any nonzero $\{c_n\}_{n\in\mathbb{N}}$ in $\ell^\infty$, say $c_m\ne 0$, we have
\begin{equation*}
  \pi\bigg(\sum_{n\in\mathbb{N}}c_nP_n\bigg)\pi(f_m(A))
  =\pi(c_mf_m(A))=c_mf_m(\pi(A))\ne 0.
\end{equation*}
It follows that $\pi|_{\mathcal{A}}$ is injective.
Therefore, the $C^*$-algebra $\pi(\mathcal{A})$ is *-isomorphic to $\ell^\infty$ and is non-separable in the norm topology.
This completes the proof.
\end{proof}

By the proof of \Cref{thm cohomology}, it seems that the condition $\dim\pi(\mathcal{A}'')=\infty$ can be relaxed to the weaker condition $\dim\pi(\mathcal{A}^{cc})=\infty$.
The following proposition indicates that these two conditions are actually equivalent.

\begin{proposition}\label{prop dim-pi(N)}
Let $\mathcal{M}$ be a separable properly infinite semifinite factor, and $\mathcal{N}$ a unital von Neumann subalgebra of $\mathcal{M}$.
Then $\dim\pi(\mathcal{N})=\dim\pi(\mathcal{N}^{cc})$.
\end{proposition}

\begin{proof}
If $\mathcal{M}$ is of type $\mathrm{I}_\infty$, then $\mathcal{N}=\mathcal{N}^{cc}$ by the von Neumann bicommutant theorem and the conclusion is clear.

According to \cite[Proposition 8.5.2, Proposition 8.5.5]{KR}, there is a normal faithful semifinite tracial weight $\tau$ on $\mathcal{M}$ such that a projection $E$ in $\mathcal{M}$ is infinite if and only if $\tau(E)=\infty$.
If $\dim\pi(\mathcal{N})=\infty$, then it follows from $\mathcal{N}\subseteq\mathcal{N}^{cc}$ that $\dim\pi(\mathcal{N}^{cc})=\infty$.

Suppose that $\dim\pi(\mathcal{N})<\infty$.
Let $T$ be a positive operator in $\mathcal{N}\cap\mathcal{K}_{\mathcal{M}}$.
If the range projection $R(T)$ is infinite in $\mathcal{M}$, then there exists a strictly decreasing sequence of positive numbers $\{a_n\}_{n\in\mathbb{N}}$ such that the spectral projection $P_n$ of $T$ with respect to the interval $(a_{n+1},a_n]$ satisfying $1\leqslant\tau(P_n)<\infty$ for each $n\in\mathbb{N}$.
In this case, there exists a sequence of mutually orthogonal projections $\{Q_n\}_{n\in\mathbb{N}}$ in $\mathcal{N}$ such that each $Q_n$ is infinite in $\mathcal{M}$.
That contradicts $\dim\pi(\mathcal{N})<\infty$.
Hence $R(T)$ is a finite projection.
Thus, for any operator $T$ in $\mathcal{N}\cap\mathcal{K}_{\mathcal{M}}$, $R(T)=R(TT^*)$ is a finite projection in $\mathcal{M}$.
Let
\begin{equation*}
  P=\bigvee\{R(T)\colon T\in\mathcal{N}\cap\mathcal{K}_{\mathcal{M}}\}\in\mathcal{N}.
\end{equation*}
By a similar argument, $P$ must be a finite projection in $\mathcal{M}$.
Since $\mathcal{N}\cap\mathcal{K}_{\mathcal{M}}$ is an ideal of $\mathcal{N}$, the projection $P$ lies in $\mathcal{N}\cap\mathcal{N}'$, the center of $\mathcal{N}$.
Therefore, we can write
\begin{equation*}
  \mathcal{N}=\mathcal{N}P\oplus\mathcal{N}(I-P).
\end{equation*}
By the definition of $P$, we have $\mathcal{N}(I-P)\cap\mathcal{K}_{\mathcal{M}}=\{0\}$, and hence $\mathcal{N}(I-P)$ is finite-dimensional.
It follows that
\begin{equation*}
  \mathcal{N}^{cc}\subseteq P\mathcal{M}P\oplus\mathcal{N}(I-P).
\end{equation*}
This gives $\pi(\mathcal{N})=\pi(\mathcal{N}^{cc})$.
\end{proof}

\section{Multiplier Algebras}\label{sec multiplier}

In this section, let $\mathcal{M}$ be a separable type $\mathrm{III}$ factor.
Note that $\mathcal{K}_{\mathcal{M}}=\{0\}$ and $\mathcal{M}$ has no nontrivial ideal.

\subsection{Multiplier Algebras}\label{subsec multiplier}

Let $\mathcal{B}$ be a type $\mathrm{I}_{\infty}$ unital subfactor of $\mathcal{M}$, $\mathcal{K}_{\mathcal{B}}$ the ideal of all compact operators in $\mathcal{B}$, and $\mathcal{N}$ the relative commutant $\mathcal{B}^c=\mathcal{B}'\cap\mathcal{M}$ of $\mathcal{B}$ in $\mathcal{M}$.
Then $\mathcal{M}$ is generated by $\mathcal{N}\cup\mathcal{B}$ as a von Neumann algebra, and
\begin{equation*}
  \mathcal{M}\cong\mathcal{N}\,\overline{\otimes}\,\mathcal{B}.
\end{equation*}
Let $\mathcal{J}$ be the $C^*$-subalgebra of $\mathcal{M}$ generated by $\mathcal{N}\mathcal{K}_{\mathcal{B}}=\{NK\colon N\in\mathcal{N},K\in\mathcal{K}_{\mathcal{B}}\}$.
Then we have
\begin{equation*}
  \mathcal{J}\cong\mathcal{N}\otimes\mathcal{K}_{\mathcal{B}}.
\end{equation*}
Here we use $\otimes$ and $\overline{\otimes}$ to represent the $C^*$-tensor product and von Neumann tensor product, respectively.
The \emph{multiplier algebra} of $\mathcal{J}$ is defined as
\begin{equation*}
  \mathcal{M}(\mathcal{J})=\{T\in\mathcal{M}\colon T\mathcal{J}\subseteq\mathcal{J},
  \mathcal{J}T\subseteq\mathcal{J}\}.
\end{equation*}
Then $\mathcal{J}$ is a closed ideal in $\mathcal{M}(\mathcal{J})$, and $\mathcal{B}$ is a subalgebra of $\mathcal{M}(\mathcal{J})$.
For more details about multiplier algebras, please refer to \cite[Chapter 2]{WO}.

Although $\mathcal{J}$ is not an ideal in $\mathcal{M}$, the following lemma shows that $\mathcal{J}$ is a hereditary $C^*$-subalgebra of $\mathcal{M}$.

\begin{lemma}\label{lem JMJ=J}
$\mathcal{J}=\mathcal{J}\mathcal{M}\mathcal{J}$.
\end{lemma}

\begin{proof}
Since $\mathcal{J}$ is a $C^*$-algebra and $I\in\mathcal{M}$, it is evident that $\mathcal{J}\subseteq\mathcal{J}\mathcal{M}\mathcal{J}$.

Let $\{E_{mn}\}_{m,n\in\mathbb{N}}$ be a system of matrix units in $\mathcal{B}$ such that $E_{00}$ is a minimal projection in $\mathcal{B}$ and $\sum_nE_{nn}=I$.
For every $A$ in $\mathcal{M}$, we set
\begin{equation*}
  A_{ij}=\sum_nE_{ni}AE_{jn}~\text{for all}~i,j\in\mathbb{N}.
\end{equation*}
Then $A_{ij}\in\mathcal{B}'\cap\mathcal{M}=\mathcal{N}$ because $E_{mn}A_{ij}=E_{mi}AE_{jn}=A_{ij}E_{mn}$ for all $m,n\in\mathbb{N}$.
It is clear that $E_{ij}\in\mathcal{K}_{\mathcal{B}}$, and therefore, $E_{ii}AE_{jj}=A_{ij}E_{ij}\in\mathcal{J}$.
Consequently,
\begin{equation*}
  E_{ii}\mathcal{M}E_{jj}\subseteq\mathcal{J}~\text{for all}~i,j\in\mathbb{N}.
\end{equation*}
Let $P_n=E_{00}+E_{11}+\cdots+E_{nn}$.
For every $A\in\mathcal{M}$ and $J_1,J_2\in\mathcal{J}$, we have
\begin{equation*}
  P_nJ_1AJ_2P_n\in P_n\mathcal{M}P_n\subseteq\mathcal{J}.
\end{equation*}
Since $J_1=\lim\limits_{n\to\infty}P_nJ_1$ and $J_2=\lim\limits_{n\to\infty}J_2P_n$ in norm topology, we conclude that
\begin{equation*}
  J_1AJ_2=\lim_{n\to\infty}P_nJ_1AJ_2P_n\in\mathcal{J}.
\end{equation*}
This completes the proof.
\end{proof}

The following result suggests that it is reasonable to consider separable $C^*$-algebras within $\mathcal{M}(\mathcal{J})$, as the latter is sufficiently large to accommodate them.

\begin{proposition}\label{prop U*AU-in-M(J)}
Let $\mathcal{M}$ be a separable type $\mathrm{III}$ factor, and $\mathcal{A}$ a separable unital $C^*$-subalgebra of $\mathcal{M}$.
Then there is a unitary operator $U$ in $\mathcal{M}$ such that $U^*\mathcal{A}U\subseteq\mathcal{M}(\mathcal{J})$.
\end{proposition}

\begin{proof}
Let $\{A_j\}_{j\in\mathbb{N}}$ be a norm-dense sequence in $\mathcal{A}$, $\{X_j\}_{j\in\mathbb{N}}$ a strong-operator dense sequence in $\mathcal{M}$, $\{Y_j\}_{j\in\mathbb{N}}$ the set of all noncommutative *-monomials generated by $\{A_j\}_{j\in\mathbb{N}}\cup\{X_j\}_{j\in\mathbb{N}}$, and $\mathcal{F}_n=\{Y_0,Y_1,\ldots,Y_n\}$ for each $n\in\mathbb{N}$.
By \Cref{lem QnFQ0=0}, there exists a sequence $\{Q_n\}_{n\in\mathbb{N}}$ of infinite projections in $\mathcal{M}$ such that $Q_n\mathcal{F}_nQ_0=\{0\}$ for every $n\geqslant 1$.
Let
\begin{equation}\label{equ R(YQ)}
  P_n=\bigvee\{R(YQ_0)\colon Y\in\mathcal{F}_n\}\leqslant I-Q_n~\text{for all}~n\in\mathbb{N}.
\end{equation}
Then $\bigvee_{n\in\mathbb{N}}P_n=I$ because the sequence $\{X_j\}_{j\in\mathbb{N}}$ generates $\mathcal{M}$ as a von Neumann algebra.
Let $E_0=P_0$, and $E_n=P_n-P_{n-1}$ for $n\geqslant 1$.
Since $P_n\ne I$, we may assume that $E_n\ne 0$ for each $n\in\mathbb{N}$ by passing to a subsequence of $\{P_n\}_{n\in\mathbb{N}}$.
Since $\mathcal{M}$ is a type $\mathrm{III}$ factor, the projections in $\{E_n\}_{n\in\mathbb{N}}$ are pairwise equivalent.

Let $\mathcal{B}_1$ be a type $\mathrm{I}_{\infty}$ unital subfactor of $\mathcal{M}$ with a system of matrix units $\{E_{mn}\}_{m,n\in\mathbb{N}}$ such that $E_{nn}=E_n$ for all $n\in\mathbb{N}$.
Let $\mathcal{K}_{\mathcal{B}_1}$ be the ideal of compact operators in $\mathcal{B}_1$, $\mathcal{N}_1$ the relative commutant $\mathcal{B}_1^c=\mathcal{B}_1'\cap\mathcal{M}$ of $\mathcal{B}_1$ in $\mathcal{M}$, and $\mathcal{J}_1$ the $C^*$-subalgebra of $\mathcal{M}$ generated by $\mathcal{N}_1\mathcal{K}_{\mathcal{B}_1}$.
For any $j,n\in\mathbb{N}$, \eqref{equ R(YQ)} shows that
\begin{equation*}
  R(Y_jP_n)\leqslant\bigvee\{R(Y_jYQ_0)\colon Y\in\mathcal{F}_n\}
  \leqslant\bigvee\{R(YQ_0)\colon Y\in\mathcal{F}_m\}=P_m
\end{equation*}
for all sufficiently large $m\in\mathbb{N}$ by the definition of $\mathcal{F}_m$.
Then by \Cref{lem JMJ=J}, we have
\begin{equation*}
  Y_jP_n=P_mY_jP_n\in\mathcal{J}_1.
\end{equation*}
It follows that $Y_j\mathcal{J}_1\subseteq\mathcal{J}_1$ since $\{P_n\}_{n\in\mathbb{N}}$ is an approximate unit of $\mathcal{J}_1$.
By the definition of $\{Y_j\}_{j\in\mathbb{N}}$, for every $j\in\mathbb{N}$, there exists $j'\in\mathbb{N}$ such that $Y_j^*=Y_{j'}$.
Hence $Y_j^*\mathcal{J}_1\subseteq\mathcal{J}_1$.
Thus, $Y_j\in\mathcal{M}(\mathcal{J}_1)$ for every $j\in\mathbb{N}$.
In particular, $A_j\in\mathcal{M}(\mathcal{J}_1)$ for every $j\in\mathbb{N}$, and therefore, $\mathcal{A}\subseteq\mathcal{M}(\mathcal{J}_1)$.
Recall that $\mathcal{B}$ is a type $\mathrm{I}_{\infty}$ unital subfactor of $\mathcal{M}$.
Hence there is a unitary operator $U$ in $\mathcal{M}$ such that $U^*\mathcal{B}_1U=\mathcal{B}$.
From this, it is straightforward to see that $U^*\mathcal{A}U\subseteq U^*\mathcal{M}(\mathcal{J}_1)U=\mathcal{M}(\mathcal{J})$.
\end{proof}

\subsection{Main Results in $\mathcal{M}(\mathcal{J})$}\label{subsec main-M(J)}

The result presented below can be derived from the proof of \cite[Proposition 2.1]{Haa}.
We will use it to prove \Cref{Lemma 3.4-again}, a comparable version of \Cref{Lemma 3.4} in the context of $\mathcal{M}(\mathcal{J})$.

\begin{proposition}\label{prop 2.1}
Let $\mathcal{M}$ be a separable type $\mathrm{III}$ factor, $\mathcal{B}_0$ a finite-dimensional unital subfactor of $\mathcal{B}$, and $\eta\colon\mathcal{B}_0\to\mathcal{M}(\mathcal{J})$ a completely positive map.
Then there exists a single operator $T$ in $\mathcal{M}(\mathcal{J})$ such that $\eta(B)=T^*BT$ for every $B\in\mathcal{B}_0$.
\end{proposition}

\begin{lemma}\label{Lemma 3.4-again}
Let $\mathcal{M}$ be a separable type $\mathrm{III}$ factor, $\mathcal{A}$ a unital $C^*$-subalgebra of $\mathcal{M}(\mathcal{J})$, and $P\in\mathcal{J}$ a projection.
Suppose $\psi\colon\mathcal{A}\to\mathcal{J}$ is a completely positive map, and there exist completely positive maps $\sigma\colon\mathcal{A}\to M_n(\mathbb{C})$ and $\eta\colon M_n(\mathbb{C})\to\mathcal{J}$ such that
\begin{enumerate}
\item [(1)] $\psi=\eta\circ\sigma$.
\item [(2)] $\sigma|_{\mathcal{A}\cap\mathcal{J}} =\psi|_{\mathcal{A}\cap\mathcal{J}}=0$.
\end{enumerate}
Then $\psi$ can be approximated in the pointwise-norm topology by maps of the form
\begin{equation*}
  A\mapsto V^*AV,
\end{equation*}
where $V\in\mathcal{J}$ and $PV=0$.
In particular, $V$ can be selected as a partial isometry such that $V^*V=\psi(I)$ when $\psi(I)$ is a projection.
\end{lemma}

\begin{proof}
Let $\mathcal{B}_0$ be a type $\mathrm{I}_n$ unital subfactor of $\mathcal{B}$ with a system of matrix units $\{E_{ij}\}_{1\leqslant i,j\leqslant n}$.
We can assume that $\sigma\colon\mathcal{A}\to\mathcal{B}_0$ and $\eta\colon\mathcal{B}_0\to\mathcal{J}$.
By \Cref{prop 2.1}, there is an operator $T$ in $\mathcal{M}(\mathcal{J})$ such that
\begin{equation*}
  \eta(B)=T^*BT \text{ for every } B\in\mathcal{B}_0.
\end{equation*}
Let $T=U|T|$ be the polar decomposition in $\mathcal{M}$.
Then $|T|\in\mathcal{J}$ as $\eta(I)=T^*T\in\mathcal{J}$.

Let $\mathcal{F}$ be a finite subset of $\mathcal{A}$ containing $I$, and $\varepsilon>0$.
We may assume that $P\in\mathcal{A}$ and $P\in\mathcal{F}$.
According to \cite[Lemma 4.4]{AH}, there are pure states $\rho^1,\rho^2,\ldots,\rho^k$ on $\mathcal{A}$ with $\rho^t|_{\mathcal{A}\cap\mathcal{J}}=0$ for $1\leqslant t\leqslant k$, and operators $A_{t,j}$ in $\mathcal{A}$ for $1\leqslant t\leqslant k,1\leqslant j\leqslant n$, such that
\begin{equation*}
  \|T\|^2\Big\|\sigma(A)-\sum_{t,i,j}\rho^t(A_{t,i}^*AA_{t,j})E_{ij}\Big\|
  <\frac{\varepsilon}{2}~\text{for every}~A\in\mathcal{F}.
\end{equation*}
It follows that
\begin{equation*}
  \Big\|\psi(A)-\sum_{t,i,j}\rho^t(A_{t,i}^*AA_{t,j})T^*E_{ij}T\Big\|
  <\frac{\varepsilon}{2}~\text{for every}~A\in\mathcal{F}.
\end{equation*}
According to \cite[Proposition 2.2]{AAP}, let $C_t$ be a positive operator in $\mathcal{A}$ with $\|C_t\|=1$ and $\rho^t(C_t)=1$ such that
\begin{equation}\label{equ excision}
  \|T\|^2\|C_t(X-\rho^t(X))C_t\|<\frac{\varepsilon}{4kn^2}
\end{equation}
for every $X\in\{A_{t,i}^*AA_{t,j}\colon 1\leqslant t\leqslant k, 1\leqslant i,j\leqslant n, A\in\mathcal{F}\}$.
By the definition of $\mathcal{K}_{\mathcal{B}}$, there exists a projection $Q$ in $\mathcal{K}_{\mathcal{B}}$ such that $\|QC_t^2Q\|>\frac{1}{2}\|C_t^2\|=\frac{1}{2}$ for all $1\leqslant t\leqslant k$.
Then there exists a nonzero spectral projection $P_t$ of $QC_t^2Q$ in $\mathcal{M}$ such that
\begin{equation}\label{equ PCP}
  P_t\geqslant P_tC_t^2P_t\geqslant\frac{1}{2}P_t~\text{for all}~1\leqslant t\leqslant k.
\end{equation}
Let $\mathcal{G}=\{C_tA_{t,i}^*AA_{t,j}C_t\colon 1\leqslant t\leqslant k, 1\leqslant i,j\leqslant n, A\in\mathcal{F}\}$.
By \Cref{lem QnFQ0=0}, there exist infinite projections $\{Q_t\}_{1\leqslant t\leqslant k}$ in $\mathcal{M}$ such that $Q_t\leqslant P_t$ for each $1\leqslant t\leqslant k$, and $Q_s\mathcal{G}Q_t=\{0\}$ when $s\ne t$.
Let $U_t$ be a partial isometry in $\mathcal{M}$ such that
\begin{equation*}
  U_t^*U_t=E_{11},\quad U_tU_t^*=Q_t.
\end{equation*}
Since \eqref{equ PCP} implies that
\begin{equation*}
  E_{11}\geqslant U_t^*Q_tC_t^2Q_tU_t\geqslant\frac{1}{2}U_t^*Q_tU_t=\frac{1}{2}E_{11},
\end{equation*}
there exists a positive operator $X_t$ in $E_{11}\mathcal{M}E_{11}$ with $\|X_t\|^2\leqslant 2$ such that
\begin{equation*}
  X_t^2(U_t^*Q_tC_t^2Q_tU_t)=(U_t^*Q_tC_t^2Q_tU_t)X_t^2=E_{11}~\text{for all}~1\leqslant t\leqslant k.
\end{equation*}
Consequently, $\rho^t(A_{t,i}^*AA_{t,j})E_{ij}
=E_{i1}X_tU_t^*Q_tC_t\rho^t(A_{t,i}^*AA_{t,j})C_tQ_tU_tX_tE_{1j}$, and then
\begin{equation}\label{equ sum-rho}
  \sum_{t,i,j}\rho^t(A_{t,i}^*AA_{t,j})T^*E_{ij}T
  =\sum_{t,i,j}T^*E_{i1}X_tU_t^*Q_tC_t\rho^t(A_{t,i}^*AA_{t,j})C_tQ_tU_tX_tE_{1j}T.
\end{equation}
Since $Q,|T|\in\mathcal{J}$, it follows from \Cref{lem JMJ=J} that
\begin{equation*}
  Q_tU_tX_tE_{1j}T=QQ_tU_tX_tE_{1j}U|T|\in\mathcal{J}.
\end{equation*}
Let $Y=\sum_{t,j}A_{t,j}C_tQ_tU_tX_tE_{1j}T\in\mathcal{J}$.
Since $Q_s\mathcal{G}Q_t=\{0\}$ when $s\ne t$, we have
\begin{equation}\label{equ Y*AY}
  Y^*AY=\sum_{t,i,j}T^*E_{i1}X_tU_t^*Q_tC_tA_{t,i}^*AA_{t,j}C_tQ_tU_tX_tE_{1j}T.
\end{equation}
From \eqref{equ excision}, \eqref{equ sum-rho} and \eqref{equ Y*AY}, it follows that
\begin{equation*}
  \Big\|\sum_{t,i,j}\rho^t(A_{t,i}^*AA_{t,j})T^*E_{ij}T-Y^*AY\Big\|
  <\frac{\varepsilon}{2}~\text{for every}~A\in\mathcal{F}.
\end{equation*}
Consequently, $\|\psi(A)-Y^*AY\|<\varepsilon$ for every $A\in\mathcal{F}$.
In particular, $\|Y^*PY\|<\varepsilon$ by the assumption $P\in\mathcal{F}$, and then we can replace $Y$ with $V=(I-P)Y\in\mathcal{J}$.
Furthermore, since $\|\psi(I)-V^*V\|<\varepsilon$, we can choose $V$ as a partial isometry such that $V^*V=\psi(I)$ when $\psi(I)$ is a projection.
\end{proof}

Now, we present the main approximation theorem for this section.
A similar conclusion can be found in \cite[Lemma 11]{EK}.

\begin{theorem}\label{thm main-M(J)}
Let $\mathcal{M}$ be a separable type $\mathrm{III}$ factor, $\mathcal{A}$ a separable unital $C^*$-subalgebra of $\mathcal{M}(\mathcal{J})$, and $P\in\mathcal{J}$ a projection.

Then for any completely positive map $\psi\colon\mathcal{A}\to\mathcal{B}$ with $\psi|_{\mathcal{A}\cap\mathcal{J}}=0$, there is a sequence $\{V_k\}_{k\in\mathbb{N}}$ in $\mathcal{M}(\mathcal{J})$ such that
\begin{enumerate}
\item [(1)]  $PV_k=0$ for every $k\in\mathbb{N}$.
\item [(2)] $\lim\limits_{k\to\infty}\|\psi(A)-V_k^*AV_k\|=0$ for every $A\in\mathcal{A}$.
\item [(3)] $\psi(A)-V_k^*AV_k\in\mathcal{J}$ for every $A\in\mathcal{A}$ and $k\in\mathbb{N}$.
\end{enumerate}
In particular, $V_k$ can be selected as a partial isometry such that $V_k^*V_k=\psi(I)$ when $\psi(I)$ is a projection.
\end{theorem}

\begin{proof}
Let $\{Q_n\}_{n\in\mathbb{N}}$ be an increasing sequence of finite rank projections in $\mathcal{K}_{\mathcal{B}}$ such that $\bigvee_{n\in\mathbb{N}}Q_n=I$, and $\{A_j\}_{j\in\mathbb{N}}$ a norm-dense sequence in $\mathcal{A}^{\mathrm{s.a.}}$ with $A_0=I$.

Fix $k\in\mathbb{N}$.
According to \cite[Theorem 2]{Arv}, there exists a sequence $\{E_n\}_{n\in\mathbb{N}}$ of finite rank positive operators in $\mathcal{K}_{\mathcal{B}}$ such that
\begin{enumerate}
\item [(1)] $\sum_nE_n^2=I$ and $\|E_nQ_m\|<2^{-n}$ for every $0\leqslant m\leqslant n-1$.
\item [(2)] $\|\psi(A_j)-\sum_nE_n\psi(A_j)E_n\|<2^{-k-1}$ for every $0\leqslant j\leqslant k$.
\item [(3)] $\psi(A)-\sum_nE_n\psi(A)E_n\in\mathcal{K}_{\mathcal{B}}$ for every $A\in\mathcal{A}$.
\end{enumerate}
Let $P_n$ denote the finite rank projection $R(E_n)$ in $\mathcal{K}_{\mathcal{B}}$.
We define a completely positive map
\begin{equation*}
  \psi_n\colon\mathcal{A}\to\mathcal{J},\quad A\mapsto P_n\psi(A)P_n.
\end{equation*}
By \Cref{Lemma 3.4-again}, we can choose a sequence $\{U_n\}_{n\in\mathbb{N}}$ in $\mathcal{J}$ inductively such that
\begin{enumerate}
\item [(1)] $PU_n=0$ for every $n\geqslant 0$, and $\|Q_mU_n\|<2^{-n}$ for every $0\leqslant m\leqslant n-1$.
\item [(2)] $\|U_m^*A_jU_n\|<2^{-2n-k-4}$ for every $0\leqslant j\leqslant n+k$ and $0\leqslant m\leqslant n-1$.
\item [(3)] $\|\psi_n(A_j)-U_n^*A_jU_n\|<2^{-n-k-3}$ for every $0\leqslant j\leqslant n+k$.
\end{enumerate}
It follows that
\begin{equation*}
  \|U_m^*A_jU_n\|<\frac{1}{2^{2\max\{m,n\}+k+4}}~\text{whenever}~0\leqslant j\leqslant\max\{m,n\}+k, m\ne n.
\end{equation*}
Let $V=\sum_nU_nE_n$.
Then
\begin{equation*}
  \sum_nE_n\psi(A_j)E_n-V^*A_jV
  =\sum_nE_n\big(\psi(A_j)-U_n^*A_jU_n\big)E_n
  -\sum_{m\ne n}E_mU_m^*A_jU_nE_n
\end{equation*}
for every $j\geqslant 0$, and $PV=0$.
The above sums on the right-hand side are norm-convergent and each summand lies in $\mathcal{J}$.
It follows that $\sum_nE_n\psi(A_j)E_n-V^*A_jV\in\mathcal{J}$ for every $j\geqslant 0$.
We also have the estimation
\begin{equation*}
  \Big\|\sum_nE_n\psi(A_j)E_n-V^*A_jV\Big\|<\frac{1}{2^{k+1}}~\text{for every}~0\leqslant j\leqslant k.
\end{equation*}
Therefore, $\psi(A_j)-V^*A_jV\in\mathcal{J}$ for every $j\geqslant 0$, and $\|\psi(A_j)-V^*A_jV\|<2^{-k}$ for every $0\leqslant j\leqslant k$.
In particular, $V^*V$ is a bounded operator if we consider $A_0=I$.
We can conclude that $V$ belongs to $\mathcal{M}$.
Furthermore, since $\|E_nQ_m\|<2^{-n}$ and $\|Q_mU_n\|<2^{-n}$ for every $n>m$, we have
\begin{equation*}
  VQ_m=\sum_nU_nE_nQ_m\in\mathcal{J},\quad Q_mV=\sum_nQ_mU_nE_n\in\mathcal{J},
\end{equation*}
for every $m\geqslant 0$.
It follows that $V\in\mathcal{M}(\mathcal{J})$.
Now we set $V_k=V$.
\end{proof}

\subsection{Voiculescu's Theorem in $\mathcal{M}(\mathcal{J})$}\label{subsec Voi-M(J)}

We now prove Voiculescu's theorem for $\mathcal{M}(\mathcal{J})$.
The proof follows a similar approach to that of \Cref{thm Voi}.

\begin{theorem}\label{thm Voi-M(J)}
Let $\mathcal{M}$ be a separable type $\mathrm{III}$ factor, and $\mathcal{A}$ a separable unital $C^*$-subalgebra of $\mathcal{M}(\mathcal{J})$.

If $\varphi\colon\mathcal{A}\to\mathcal{B}$ is a unital *-homomorphism with $\varphi|_{\mathcal{A}\cap\mathcal{J}}=0$, then there is a sequence $\{V_k\}_{k\in\mathbb{N}}$ of isometries in $\mathcal{M}(\mathcal{J})\otimes M_2(\mathbb{C})$ such that
\begin{enumerate}
\item [(1)] $V_k^*V_k=I\oplus I,V_kV_k^*=I\oplus 0$ for every $k\in\mathbb{N}$.
\item [(2)] $\lim\limits_{k\to\infty}\|(A\oplus\varphi(A))-V_k^*(A\oplus 0)V_k\|=0$ for every $A\in\mathcal{A}$.
\item [(3)] $(A\oplus\varphi(A))-V_k^*(A\oplus 0)V_k\in\mathcal{J}\otimes M_2(\mathbb{C})$ for every $A\in\mathcal{A}$ and $k\in\mathbb{N}$.
\end{enumerate}
\end{theorem}

\begin{proof}
Let $\{E_{mn}\}_{m,n\in\mathbb{N}}$ be a system of matrix units in $\mathcal{B}$ such that $\sum_nE_{nn}=I$ and $E_{00}$ is an infinite projection in $\mathcal{B}$.
Let $T$ be an isometry in $\mathcal{B}$ with $T^*T=I$ and $TT^*=E_{00}$, and let $S$ denote the isometry $\sum_nE_{n+1,n}$ in $\mathcal{B}$.
We define
\begin{equation*}
  \psi\colon\mathcal{A}\to\mathcal{B},\quad A\mapsto\sum_nE_{n0}T\varphi(A)T^*E_{0n}.
\end{equation*}
By \Cref{thm main-M(J)}, there is a sequence $\{U_k\}_{k\in\mathbb{N}}$ of isometries in $\mathcal{M}(\mathcal{J})$ such that
\begin{equation*}
  \lim_{k\to\infty}\|U_k\psi(A)-AU_k\|=0~\text{for every}~A\in\mathcal{A},
\end{equation*}
and $U_k\psi(A)-AU_k\in\mathcal{J}$ for every $A\in\mathcal{A}$ and $k\in\mathbb{N}$.
The rest of the proof mirrors that of \Cref{thm Voi}.
\end{proof}

\subsection{Applications in $\mathcal{M}(\mathcal{J})$}\label{subsec application-M(J)}

Let $T$ be an operator in $\mathcal{M}(\mathcal{J})$.
We say that $T$ is \emph{reducible} in $\mathcal{M}(\mathcal{J})$ if there is a projection $P\in\mathcal{M}(\mathcal{J})$ such that $PT=TP$ and $P\ne 0,I$.
Similar to \Cref{thm reducible}, \Cref{thm Voi-M(J)} implies the following denseness result.

\begin{theorem}\label{thm reducible2}
Let $\mathcal{M}$ be a separable type $\mathrm{III}$ factor.
Then the set of all reducible operators is norm-dense in $\mathcal{M}(\mathcal{J})$.
\end{theorem}

If $\mathscr{A}$ is a separable unital subalgebra of $\mathcal{M}(\mathcal{J})/\mathcal{J}$, then the \emph{essential lattice} $Lat_e(\mathscr{A})$ of $\mathscr{A}$ is the set of all projections $p$ in $\mathcal{M}(\mathcal{J})/\mathcal{J}$ such that $p^{\perp}ap=0$ for every $a\in\mathscr{A}$.
Similar to \Cref{lem distance}, \Cref{thm Voi-M(J)} implies the following distance formula.

\begin{lemma}\label{lem distance2}
Let $\mathcal{M}$ be a separable type $\mathrm{III}$ factor, and $\mathscr{A}$ a separable unital subalgebra of $\mathcal{M}(\mathcal{J})/\mathcal{J}$.
Then for any $t$ in $\mathcal{M}(\mathcal{J})/\mathcal{J}$, there is a projection $q$ in $Lat_e(\mathscr{A})$ such that
\begin{equation*}
  \|q^{\perp}tq\|=\mathrm{dist}(t,\mathscr{A}).
\end{equation*}
\end{lemma}

Note that every separable norm-closed unital subalgebra of $\mathcal{M}(\mathcal{J})/\mathcal{J}$ is reflexive by \Cref{lem distance2}.
In particular, the following generalization of Voiculescu's relative bicommutant theorem holds (see \Cref{thm bicommutant}).

\begin{theorem}\label{thm bicommutant2}
Let $\mathcal{M}$ be a separable type $\mathrm{III}$ factor.
Then every separable unital $C^*$-subalgebra of $\mathcal{M}(\mathcal{J})/\mathcal{J}$ equals its relative bicommutant.
\end{theorem}

Let $\mathcal{A}$ be a unital $C^*$-subalgebra of $\mathcal{B}$.
Since the relative commutant of $\mathcal{A}$ in $\mathcal{M}(\mathcal{J})$ contains $\mathcal{N}\cup(\mathcal{A}'\cap\mathcal{B})$, the relative bicommutant of $\mathcal{A}$ in $\mathcal{M}(\mathcal{J})$ is contained in $\mathcal{A}''$.
Clearly, the relative bicommutant of $\mathcal{A}$ in $\mathcal{M}(\mathcal{J})$ contains $\mathcal{A}''$.
Thus, the relative bicommutant of $\mathcal{A}$ in $\mathcal{M}(\mathcal{J})$ is $\mathcal{A}''$.

Let $\pi\colon\mathcal{M}(\mathcal{J})\to\mathcal{M}(\mathcal{J})/\mathcal{J}$ be the canonical quotient map.
By definition, it is easy to see that $\mathcal{J}\cap\mathcal{B}=\mathcal{K}_{\mathcal{B}}$.
From this, a similar version of \Cref{prop pi(N)} holds for unital von Neumann subalgebras of $\mathcal{B}\subseteq\mathcal{M}(\mathcal{J})$.
Therefore, we obtain the following cohomological result by \Cref{thm bicommutant2} (see \Cref{thm cohomology}).

\begin{theorem}\label{thm cohomology2}
Let $\mathcal{M}$ be a separable type $\mathrm{III}$ factor, and $\mathcal{A}$ a separable infinite-dimensional unital $C^*$-subalgebra of $\mathcal{B}$.
Then $H^1(\mathcal{A},\mathcal{J})\ne\{0\}$.
\end{theorem}

\section{Nuclear Length}\label{sec nuc-L}

\subsection{Nuclear Length}\label{subsec nuc-L}

Let $\mathcal{M}$ be a separable properly infinite factor, and $\mathcal{B}$ a $C^*$-subalgebra of $\mathcal{M}$.
Since the class $\widehat{\mathfrak{SF}}$ of completely positive maps introduced in \Cref{def SF} is very important in Voiculescu's theorem (see \Cref{thm Voi}), we will present a generalization of \Cref{prop into-typeI} in this section.
Inspired by quasicentral approximate units, we introduce the \emph{nuclear length} of $\mathcal{B}$ in $\mathcal{M}$.

\begin{definition}\label{def nuc-L}
We set $L_{\mathrm{nuc}}(\mathcal{B},\mathcal{M})=0$ if $\mathcal{B}$ is nuclear.
Inductively, we set
\begin{equation*}
  L_{\mathrm{nuc}}(\mathcal{B},\mathcal{M})=m,
\end{equation*}
if $L_{\mathrm{nuc}}(\mathcal{B},\mathcal{M})\ne k$ for every $0\leqslant k\leqslant m-1$, and for any finite subset $\mathcal{F}$ of $\mathcal{B}$ and any $\varepsilon>0$,
there exists a sequence $\{E_n\}_{n\in\mathbb{N}}$ of positive operators in $\mathcal{M}$ and a sequence $\{\mathcal{B}_n\}_{n\in\mathbb{N}}$ of $C^*$-subalgebras of $\mathcal{M}$ such that
\begin{enumerate}
\item [(1)] $\sum_nE_n^2=I$, and $L_{\mathrm{nuc}}(\mathcal{B}_n,\mathcal{M})\leqslant m-1$ for every $n\in\mathbb{N}$.
\item [(2)] $E_n\mathcal{B}E_n\subseteq\mathcal{B}_n$ for every $n\in\mathbb{N}$.
\item [(3)] $\|B-\sum_nE_nBE_n\|<\varepsilon$ for every $B\in\mathcal{F}$.
\end{enumerate}
\end{definition}

It is evident from the above definition that $L_{\mathrm{nuc}}(U^*\mathcal{B}U,\mathcal{M})=L_{\mathrm{nuc}}(\mathcal{B},\mathcal{M})$ for every unitary operator $U$ in $\mathcal{M}$.
Consequently, the nuclear length is unitarily invariant.

Let $P_{\mathcal{B}}=\bigvee_{B\in\mathcal{B}}R(B)$, where $R(B)$ is the range projection of $B$.
The \emph{multiplier algebra} of $\mathcal{B}$ is then defined as
\begin{equation*}
  \mathcal{M}(\mathcal{B})=\{T\in P_{\mathcal{B}}\mathcal{M}P_{\mathcal{B}}\colon T\mathcal{B}\subseteq\mathcal{B},
  \mathcal{B}T\subseteq\mathcal{B}\}.
\end{equation*}
Note that $\mathcal{B}$ is an ideal of $\mathcal{M}(\mathcal{B})$ and $P_{\mathcal{B}}$ is the identity of $\mathcal{M}(\mathcal{B})$.

\begin{lemma}\label{lem nuc-L-M(B)}
If $L_{\mathrm{nuc}}(\mathcal{B},\mathcal{M})<\infty$, then $L_{\mathrm{nuc}}(\mathcal{M}(\mathcal{B}),\mathcal{M})\leqslant 1+L_{\mathrm{nuc}}(\mathcal{B},\mathcal{M})$.
\end{lemma}

\begin{proof}
Let $\mathcal{F}$ be a finite subset of $\mathcal{M}(\mathcal{B})$, and $\varepsilon>0$.
According to \cite[Theorem 2]{Arv}, there is a sequence $\{E_n\}_{n=1}^{\infty}$ of positive operators in $\mathcal{B}$ such that $\sum_{n=1}^{\infty}E_n^2=P_{\mathcal{B}}$ and
\begin{equation*}
  \Big\|B-\sum_{n=1}^{\infty}E_nBE_n\Big\|<\varepsilon~\text{for every}~B\in\mathcal{F}.
\end{equation*}
We set $E_0=I-P_{\mathcal{B}}$, and $\mathcal{B}_n=\mathcal{B}$ for every $n\in\mathbb{N}$.
\end{proof}

Let $\mathcal{B}$ be a type $\mathrm{I}_{\infty}$ unital subfactor of $\mathcal{M}$, and $\mathcal{K}_{\mathcal{B}}$ the ideal generated by finite rank projections in $\mathcal{B}$.
It is well-known that $\mathcal{K}_{\mathcal{B}}$ is nuclear while $\mathcal{B}$ is not.
Since $\mathcal{B}$ is the multiplier algebra of $\mathcal{K}_{\mathcal{B}}$, we have $L_{\mathrm{nuc}}(\mathcal{B},\mathcal{M})=1$ by \Cref{lem nuc-L-M(B)}.

\begin{example}
If $\mathcal{B}$ is a von Neumann algebra of type $\mathrm{I}$, then $L_{\mathrm{nuc}}(\mathcal{B},\mathcal{M})\leqslant 1$.
\end{example}

\begin{proof}
There is a sequence $\{\mathcal{A}_n\}_{n\in\mathbb{N}}$ of abelian von Neumann algebras such that
\begin{equation*}
  \mathcal{B}=\big(\mathcal{A}_0\,\overline{\otimes}\,\mathcal{B}(\ell^2)\big)
  \bigoplus
  \sideset{}{^{\oplus}}\prod_{n=1}^{\infty}\mathcal{A}_n\otimes M_n(\mathbb{C}).
\end{equation*}
Let
\begin{equation*}
  \mathcal{B}_0=\big(\mathcal{A}_0\otimes\mathcal{K}(\ell^2)\big)
  \bigoplus\sideset{}{^{\oplus}}\sum_{n=1}^{\infty}\mathcal{A}_n\otimes M_n(\mathbb{C}).
\end{equation*}
Since $\mathcal{B}_0$ is nuclear and $\mathcal{B}=\mathcal{M}(\mathcal{B}_0)$, we get $L_{\mathrm{nuc}}(\mathcal{B},\mathcal{M})\leqslant 1$ by \Cref{lem nuc-L-M(B)}.
\end{proof}

The following theorem is a generalization of \Cref{prop into-typeI}.

\begin{theorem}\label{thm L(B)<infty}
Let $\mathcal{M}$ be a separable properly infinite factor, $\mathcal{A}$ a unital $C^*$-subalgebra of $\mathcal{M}$, and $\mathcal{B}$ a $C^*$-subalgebra of $\mathcal{M}$ with $L_{\mathrm{nuc}}(\mathcal{B},\mathcal{M})<\infty$.

Then $\psi\in\widehat{\mathfrak{SF}}$ for every completely positive map $\psi\colon\mathcal{A}\to\mathcal{B}$ with $\psi|_{\mathcal{A}\cap\mathcal{K}_{\mathcal{M}}}=0$.
\end{theorem}

\begin{proof}
Induction on $L_{\mathrm{nuc}}(\mathcal{B},\mathcal{M})=m$ is performed.
If $\mathcal{B}$ is nuclear, then the inclusion map $\mathrm{id}_{\mathcal{B}}\colon\mathcal{B}\hookrightarrow\mathcal{M}$ is nuclear.
Therefore, the composition $\psi=\mathrm{id}_{\mathcal{B}}\circ\psi$ is a nuclear map with respect to $\mathcal{K}_{\mathcal{M}}$, and thus $\psi\in\widehat{\mathfrak{F}}\subseteq\widehat{\mathfrak{SF}}$.

Assume that $m\geqslant 1$.
Let $\mathcal{F}$ be a finite subset of $\mathcal{A}$ containing $I$, and $\varepsilon>0$.
By \Cref{def nuc-L}, we can find $\{E_n\}_{n\in\mathbb{N}}$ and $\{\mathcal{B}_n\}_{n\in\mathbb{N}}$ such that
\begin{enumerate}
\item [(1)] $\sum_nE_n^2=I$, and $L_{\mathrm{nuc}}(\mathcal{B}_n,\mathcal{M})\leqslant m-1$ for every $n\in\mathbb{N}$.
\item [(2)] $E_n\mathcal{B}E_n\subseteq\mathcal{B}_n$ for every $n\in\mathbb{N}$.
\item [(3)] $\|\psi(A)-\sum_nE_n\psi(A)E_n\|<\varepsilon$ for every $A\in\mathcal{F}$.
\end{enumerate}
By induction, the completely positive map $\psi_n\colon\mathcal{A}\to\mathcal{B}_n$ defined by $A\mapsto E_n\psi(A)E_n$ lies in $\widehat{\mathfrak{SF}}$, and
\begin{equation*}
  \Big\|\psi(A)-\sum_n\psi_n(A)\Big\|<\varepsilon~\text{for every}~A\in\mathcal{F}.
\end{equation*}
Then $\sum_n\psi_n(I)$ converges in the strong-operator topology since $I\in\mathcal{F}$.
It follows that $\sum_n\psi_n\in\widehat{\mathfrak{SF}}$ by \Cref{lem SF-addition}.
Therefore, $\psi\in\widehat{\mathfrak{SF}}$.
\end{proof}

\subsection{Approximate Nuclear Length}\label{subsec nuc-AL}
At last, we introduce the \emph{approximate nuclear length}.
Let $\mathcal{M}$ be a separable properly infinite factor, and $\mathcal{B}$ a $C^*$-subalgebra of $\mathcal{M}$.

\begin{definition}
We set $AL_{\mathrm{nuc}}(\mathcal{B},\mathcal{M})=0$ if the inclusion map $\mathrm{id}_{\mathcal{B}}\colon\mathcal{B}\to\mathcal{M}$ is nuclear.
Inductively, we set
\begin{equation*}
  AL_{\mathrm{nuc}}(\mathcal{B},\mathcal{M})=m,
\end{equation*}
if $AL_{\mathrm{nuc}}(\mathcal{B},\mathcal{M})\ne k$ for every $0\leqslant k\leqslant m-1$,
and for any finite subset $\mathcal{F}$ of $\mathcal{B}$ and any $\varepsilon>0$, there is a sequence $\{\mathcal{B}_n\}_{n\in\mathbb{N}}$ of $C^*$-subalgebras of $\mathcal{M}$, and a sequence $\{\psi_n\colon\mathcal{B}\to\mathcal{B}_n\}_{n\in\mathbb{N}}$ of completely positive maps such that
\begin{enumerate}
\item [(1)] $AL_{\mathrm{nuc}}(\mathcal{B}_n,\mathcal{M})\leqslant m-1$ for every $n\in\mathbb{N}$.
\item [(2)] $\|B-\sum_n\psi_n(B)\|<\varepsilon$ for every $B\in\mathcal{F}$.
\end{enumerate}
\end{definition}

By definition, it is clear that $AL_{\mathrm{nuc}}(\mathcal{B},\mathcal{M})\leqslant L_{\mathrm{nuc}}(\mathcal{B},\mathcal{M})$ and the approximate nuclear length is unitarily invariant.
Let $\pi_1,\pi_2\colon\mathcal{B}\to\mathcal{M}$ be *-homomorphisms.
We say that $\pi_1$ and $\pi_2$ are \emph{approximately unitarily equivalent} (denoted by $\pi_1\sim_a\pi_2$) if for any finite subset $\mathcal{F}$ of $\mathcal{B}$ and any $\varepsilon>0$, there is a unitary operator $U$ in $\mathcal{M}$ such that
\begin{equation*}
  \|\pi_1(A)-U^*\pi_2(A)U\|<\varepsilon~\text{for every}~A\in\mathcal{F}.
\end{equation*}
Obviously, $\pi_1\sim_a\pi_2$ implies that $\ker\pi_1=\ker\pi_2$.
Recall that $\mathrm{id}_{\mathcal{B}}\colon\mathcal{B}\hookrightarrow\mathcal{M}$ is the inclusion map.
The following result shows that the approximate nuclear length is approximately unitarily invariant.

\begin{lemma}\label{lem AL-pi(B)}
Let $\mathcal{M}$ be a separable properly infinite factor, and $\mathcal{B}$ a $C^*$-subalgebra of $\mathcal{M}$.
If $\pi\colon\mathcal{B}\to\mathcal{M}$ is a *-homomorphism with $\pi\sim_a\mathrm{id}_{\mathcal{B}}$, then
\begin{equation*}
  AL_{\mathrm{nuc}}(\pi(\mathcal{B}),\mathcal{M})
  =AL_{\mathrm{nuc}}(\mathcal{B},\mathcal{M}).
\end{equation*}
\end{lemma}

\begin{proof}
Note that $\pi$ is faithful since $\mathrm{id}_{\mathcal{B}}$ is.
Let $\mathcal{F}$ be a finite subset of $\mathcal{B}$, and $\varepsilon>0$.
Then there is a unitary operator $U$ in $\mathcal{M}$ such that
\begin{equation*}
  \|\pi(B)-U^*BU\|<\frac{\varepsilon}{2}~\text{for every}~B\in\mathcal{F}.
\end{equation*}

If $AL_{\mathrm{nuc}}(\mathcal{B})=0$, i.e., the inclusion map $\mathrm{id}_{\mathcal{B}}\colon\mathcal{B}\hookrightarrow\mathcal{M}$ is nuclear, then there exists a factorable map $\psi\colon\mathcal{B}\to\mathcal{M}$ such that $\|B-\psi(B)\|<\frac{\varepsilon}{2}$ for every $B\in\mathcal{F}$.
It follows that
\begin{equation*}
  \|\pi(B)-U^*\psi(B)U\|<\varepsilon~\text{for every}~B\in\mathcal{F}.
\end{equation*}
Let $\varphi\colon\pi(\mathcal{B})\to\mathcal{M}$ be the factorable map defined by $\pi(B)\mapsto U^*\psi(B)U$.
Then
\begin{equation*}
  \|\pi(B)-\varphi(\pi(B))\|<\varepsilon~\text{for every}~B\in\mathcal{F}.
\end{equation*}
Hence $\mathrm{id}_{\pi(\mathcal{B})}$ is nuclear.

If $AL_{\mathrm{nuc}}(\mathcal{B})=m\geqslant 1$, then we can find $\{\mathcal{B}_n\}_{n\in\mathbb{N}}$ and $\{\psi_n\colon\mathcal{B}\to\mathcal{B}_n\}_{n\in\mathbb{N}}$ such that
\begin{enumerate}
\item [(1)] $AL_{\mathrm{nuc}}(\mathcal{B}_n,\mathcal{M})\leqslant m-1$ for every $n\in\mathbb{N}$.
\item [(2)] $\|B-\sum_n\psi_n(B)\|<\frac{\varepsilon}{2}$ for every $B\in\mathcal{F}$.
\end{enumerate}
Let $\mathcal{A}_n=U^*\mathcal{B}_nU$, and $\varphi_n\colon\pi(\mathcal{B})\to\mathcal{A}_n,\pi(B)\mapsto U^*\psi_n(B)U$.
Then
\begin{enumerate}
\item [(1)] $AL_{\mathrm{nuc}}(\mathcal{A}_n,\mathcal{M})\leqslant m-1$ for every $n\in\mathbb{N}$.
\item [(2)] $\|\pi(B)-\sum_n\varphi_n(\pi(B))\|<\varepsilon$ for every $B\in\mathcal{F}$.
\end{enumerate}
Hence $AL_{\mathrm{nuc}}(\pi(\mathcal{B}),\mathcal{M})\leqslant AL_{\mathrm{nuc}}(\mathcal{B},\mathcal{M})$.
Conversely, $AL_{\mathrm{nuc}}(\mathcal{B},\mathcal{M})\leqslant AL_{\mathrm{nuc}}(\pi(\mathcal{B}),\mathcal{M})$ since $\pi^{-1}\sim_a\mathrm{id}_{\pi(\mathcal{B})}$.
This completes the proof.
\end{proof}

Similar to \Cref{thm L(B)<infty}, we have the following result.

\begin{theorem}\label{thm AL(B)<infty}
Let $\mathcal{M}$ be a separable properly infinite factor, $\mathcal{A}$ a unital $C^*$-subalgebra of $\mathcal{M}$, and $\mathcal{B}$ a $C^*$-subalgebra of $\mathcal{M}$ with $AL_{\mathrm{nuc}}(\mathcal{B},\mathcal{M})<\infty$.

Then $\psi\in\widehat{\mathfrak{SF}}$ for every completely positive map $\psi\colon\mathcal{A}\to\mathcal{B}$ with $\psi|_{\mathcal{A}\cap\mathcal{K}_{\mathcal{M}}}=0$.
\end{theorem}

\end{document}